\documentclass[12pt,a4paper]{amsart}
\usepackage{amssymb}
\usepackage{amsmath, amscd}
\usepackage{amsthm}
\usepackage{mathtools}
\usepackage{mathrsfs}
\usepackage{booktabs}
\usepackage{perpage}
\usepackage[all,cmtip]{xy}
\usepackage{setspace}
\usepackage{eucal}
\usepackage{amsbsy}
\usepackage[T1]{fontenc}
\usepackage{verbatim}
\usepackage{mathdots}
\usepackage{mathrsfs}
\usepackage{ifthen}
\usepackage{blkarray}
\usepackage{multirow}
\usepackage{hyperref}
\usepackage{enumerate}
\usepackage[dvipsnames]{xcolor}
\usepackage{tikz}
\usepackage{fullpage}
\usepackage{colonequals}
\usepackage{braket}
\usepackage{lipsum}
\usepackage{float}

\usepackage{graphicx}

\makeatletter
\newcommand{\sigmaop}[1]{\mathop{\mathpalette\@sigmaop{#1}}\slimits@}
\newcommand{\@sigmaop}[2]{%
  \vphantom{\sum}%
  \sbox\z@{$\m@th#1\sum$}%
  \dimen@=\ht\z@ \advance\dimen@\dp\z@
  \dimen\tw@=\wd\z@
  \ifx#1\displaystyle\dimen@=.9\dimen@\fi
  \ooalign{%
    \hidewidth
    $\vcenter{\hbox{$\m@th#1#2$\kern.3\dimen\tw@}%
     \ifx#1\scriptstyle\kern-.25ex\fi}$\hidewidth\cr
    $\vcenter{\hbox{%
      \resizebox{!}{\dimen@}{$\m@th\boxtimes$}%
    }\ifx#1\scriptstyle\kern-.25ex\fi}$\cr
  }%
}
\makeatother

\mathtoolsset{showonlyrefs=true}

\numberwithin{equation}{subsection}

\newtheorem{theorem}{Theorem}[subsection]
\newtheorem{lemma}[theorem]{Lemma}

\newtheorem{conjecture}[theorem]{Conjecture}

\newtheorem{corollary}[theorem]{Corollary}
\newtheorem{definition}[theorem]{Definition}
\newtheorem{question}[theorem]{Question}

\newtheorem{proposition}[theorem]{Proposition}

\newtheorem{assumption}[theorem]{Assumption}

\newtheorem*{thm1}{Theorem 1}
\newtheorem*{thm2}{Theorem 2}
\newtheorem*{thm3}{Theorem 3}

\theoremstyle{remark}

\newtheorem{rmk}[theorem]{Remark}

\newcommand{\GZip}{\mathop{\text{$G$-{\tt Zip}}}\nolimits}

\newcommand{\HZip}{\mathop{\text{$H$-{\tt Zip}}}\nolimits}

\newcommand{\GF}{\mathop{\text{$G$-{\tt ZipFlag}}}\nolimits}

\newskip\procskipamount
\newskip\interskipamount
\newskip\refskipamount
\newcommand{\procskip}{\vskip\procskipamount}
\newcommand{\interskip}{\vskip\interskipamount}
\newcommand{\refskip}{\vskip\refskipamount}

\newcommand{\procbreak}{\par
   \ifdim\lastskip<\procskipamount\removelastskip
   \penalty-100
   \procskip\fi
   \noindent\ignorespaces}

\newcommand{\titlebreak}{\par%
\ifdim\lastskip<\interskipamount\removelastskip%
\penalty10000%
\interskip\fi%
\noindent}%

\newcommand{\interbreak}{\par%
\ifdim\lastskip<\interskipamount\removelastskip%
\penalty-100%
\interskip\fi%
\noindent\ignorespaces}%

\newcommand{\refbreak}{\par%
\ifdim\lastskip<\refskipamount\removelastskip%
\penalty-100%
\refskip\fi%
\noindent\ignorespaces}%

\newcounter{listcounter}
\newcounter{deflistcounter}
\newcounter{equivcounter}

\newskip{\itemsepamount}
\newskip{\topsepamount}
\newenvironment{assertionlist}{%
  \begin{list}
    {\upshape (\arabic{listcounter})}
    {\setlength{\leftmargin}{18pt}
     \setlength{\rightmargin}{0pt}
     \setlength{\itemindent}{0pt}
     \setlength{\labelsep}{5pt}
     \setlength{\labelwidth}{13pt}
     \setlength{\listparindent}{\parindent}
     \setlength{\parsep}{0pt}
     \setlength{\itemsep}{\itemsepamount}
     \setlength{\topsep}{\topsepamount}
     \usecounter{listcounter}}}
  {\end{list}}

\newenvironment{definitionlist}{%
  \begin{list}
    {\upshape (\alph{deflistcounter})}
    {\setlength{\leftmargin}{18pt}
     \setlength{\rightmargin}{0pt}
     \setlength{\itemindent}{0pt}
     \setlength{\labelsep}{5pt}
     \setlength{\labelwidth}{13pt}
     \setlength{\listparindent}{\parindent}
     \setlength{\parsep}{0pt}
     \setlength{\itemsep}{\itemsepamount}
     \setlength{\topsep}{\topsepamount}
     \usecounter{deflistcounter}}}
  {\end{list}}

\newenvironment{equivlist}{%
  \begin{list}
    {\upshape (\roman{equivcounter})}
    {\setlength{\leftmargin}{18pt}
     \setlength{\rightmargin}{0pt}
     \setlength{\itemindent}{0pt}
     \setlength{\labelsep}{5pt}
     \setlength{\labelwidth}{13pt}
     \setlength{\listparindent}{\parindent}
     \setlength{\parsep}{0pt}
     \setlength{\itemsep}{\itemsepamount}
     \setlength{\topsep}{\topsepamount}
     \usecounter{equivcounter}}}
  {\end{list}}

\newenvironment{bulletlist}{%
  \begin{list}
    {\upshape \textbullet}
    {\setlength{\leftmargin}{18pt}
     \setlength{\rightmargin}{0pt}
     \setlength{\itemindent}{0pt}
     \setlength{\labelsep}{6pt}
     \setlength{\labelwidth}{12pt}
     \setlength{\listparindent}{\parindent}
     \setlength{\parsep}{0pt}
     \setlength{\itemsep}{\itemsepamount}
     \setlength{\topsep}{\topsepamount}}}
  {\end{list}}

% ...

\newcommand{\Acal}{{\mathcal A}}

\newcommand{\Ccal}{{\mathcal C}}

\newcommand{\Hcal}{{\mathcal H}}
\newcommand{\Ical}{{\mathcal I}}

\newcommand{\Lcal}{{\mathcal L}}

\newcommand{\Ocal}{{\mathcal O}}

\newcommand{\Scal}{{\mathcal S}}

\newcommand{\Ucal}{{\mathcal U}}
\newcommand{\Vcal}{{\mathcal V}}

\newcommand{\Xcal}{{\mathcal X}}

\newcommand{\Zcal}{{\mathcal Z}}

\newcommand{\pfr}{{\mathfrak p}}

\newcommand{\sfr}{{\mathfrak s}}

\newcommand{\Sfr}{{\mathfrak S}}

\renewcommand{\AA}{\mathbb{A}}
\newcommand{\BB}{\mathbb{B}}
\newcommand{\CC}{\mathbb{C}}

\newcommand{\FF}{\mathbb{F}}
\newcommand{\GG}{\mathbb{G}}

\newcommand{\NN}{\mathbb{N}}

\newcommand{\QQ}{\mathbb{Q}}
\newcommand{\RR}{\mathbb{R}}

\newcommand{\TT}{\mathbb{T}}

\newcommand{\ZZ}{\mathbb{Z}}

\newcommand{\Ascr}{{\mathscr A}}

\newcommand{\cent}{{\rm Cent}}

\DeclareMathOperator{\Gal}{Gal}

\DeclareMathOperator{\Span}{Span}

\DeclareMathOperator{\Lie}{Lie}

\DeclareMathOperator{\ord}{ord}

\DeclareMathOperator{\Stab}{Stab}

\DeclareMathOperator{\pr}{pr}

\DeclareMathOperator{\Sh}{Sh}
\DeclareMathOperator{\spec}{Spec}

\DeclareMathOperator{\Sym}{Sym}

\DeclareMathOperator{\GL}{GL}

\DeclareMathOperator{\GSp}{GSp}

\DeclareMathOperator{\Sp}{Sp}
\DeclareMathOperator{\U}{U}
\DeclareMathOperator{\GU}{GU}

\newcommand{\gx}{(\mathbf G, \mathbf X)}

\newcommand{\End}{{\rm End}}

\newcommand{\loccit}{{\em loc.\ cit. }}
\newcommand{\loccitn}{{\em loc.\ cit.}}

\newcommand{\diag}{{\rm diag}}

\DeclareMathOperator{\Tr}{Tr}

\DeclareMathOperator{\val}{val}

\DeclareMathOperator{\appro}{approx}

\DeclareMathOperator{\ind}{ind}

\DeclareMathOperator{\iden}{id}
\DeclareMathOperator{\Norm}{Norm}

\DeclareMathOperator{\Res}{Res}
\DeclareMathOperator{\Flag}{Flag}

\DeclareMathOperator{\Frac}{Frac}

\DeclareMathOperator{\Cox}{Cox}
\DeclareMathOperator{\triv}{triv}

\DeclareMathOperator{\van}{van}

\DeclareMathOperator{\aut}{aut}

\newcommand{\relmiddle}[1]{\mathrel{}\middle#1\mathrel{}}

\newcommand{\xdasharrow}[2][->]{
\tikz[baseline=-\the\dimexpr\fontdimen22\textfont2\relax]{
\node[anchor=south,font=\scriptsize, inner ysep=1.5pt,outer xsep=2.2pt](x){#2};
\draw[shorten <=3.4pt,shorten >=3.4pt,dashed,#1](x.south west)--(x.south east);
}
}

\newcommand{\Ha}{\mathsf{Ha}}

\newcommand{\pha}{\mathsf{pHa}}
\newcommand{\zip}{\mathsf{zip}}
\newcommand{\GS}{\mathsf{GS}}

\newcommand{\flag}{\mathsf{flag}}

\begin{document}

\title{The cone conjecture for vector-valued Siegel automorphic forms} 

\author{Jean-Stefan Koskivirta}

\date{}

\maketitle

\begin{abstract}
We prove that the effective cone of automorphic vector bundles on the Siegel modular variety $\Acal_n$ in characteristic $p$ at a place of good reduction is encoded by the stack of $G$-zips of Pink--Wedhorn--Ziegler. Specifically, we show that the degree zero cohomology groups of automorphic vector bundles always vanish outside of the zip cone. This result is a special case of a general conjecture formulated by the Goldring and the author for all Hodge-type Shimura varieties of good reduction. In the case $n=3$, we give explicit conditions for the vanishing of the $0$-th cohomology group. Finally, in the course of the proof we define the notion of automorphic forms of trivial-type and study their properties.
\end{abstract}

\section{Introduction}

In this paper, we study vector-valued Siegel automorphic forms in characteristic $p$. We verify a conjecture about the weights of such automorphic forms, that was first proposed by Goldring and the author in \cite{Goldring-Koskivirta-global-sections-compositio} for general Shimura varieties of Hodge-type. Roughly speaking, it asserts that the vanishing of the cohomology in degree zero of automorphic vector bundles is controlled by the stack of Pink--Wedhorn--Ziegler (\cite{Pink-Wedhorn-Ziegler-zip-data}). The first instance of this type of vanishing result was the case of Hilbert--Blumenthal Shimura varieties, carried out in \cite{Goldring-Koskivirta-global-sections-compositio}, and independently by Diamond--Kassaei in \cite{Diamond-Kassaei}. We briefly review the results in this case. Let $X_{\mathbf{F}}$ be the special fiber at a prime $p$ of good reduction of a Hilbert--Blumenthal Shimura variety, attached to a totally real field extension $\mathbf{F}/\QQ$. Let $\Sigma$ denote the set of real embeddings $\tau\colon \mathbf{F}\to \RR$. Write $\Omega$ for the Hodge vector bundle of $X_{\mathbf{F}}$. The action of $\Ocal_{\mathbf{F}}$ yields a decomposition $\Omega=\bigoplus_{\tau\in \Sigma} \omega_\tau$ into line bundles. To any tuple $\underline{k}=(k_\tau)_{\tau\in \Sigma}\in \ZZ^\Sigma$, attach the line bundle $\omega^{\underline{k}}\colonequals \bigotimes_{\tau \in \Sigma}\omega_{\tau}^{k_\tau}$ on $X_{\mathbf{F}}$. We define a cone $\Ccal_{\pha}\subset \ZZ^\Sigma$ as the cone spanned by the weights of partial Hasse invariants constructed by Andreatta--Goren. The vanishing theorem of \cite{Goldring-Koskivirta-global-sections-compositio} states that for any $\underline{k}$ which lies in the complement of $\Ccal_{\pha}$, the space $H^0(X_{\mathbf{F}},\omega^{\underline{k}})$ vanishes. In particular, when $p$ is split in $\mathbf{F}$, this vanishing result takes the following form:
\begin{equation}\label{neg-intro}\tag{1}
    H^0(X_{\mathbf{F}}, \omega^{\underline{k}})\neq 0 \ \Longleftrightarrow \ \forall \tau\in \Sigma, \ k_\tau\leq 0.
\end{equation}

In the present article, we prove an analogue of these statements for the family of Siegel modular varieties $\Acal_n$ (for any $n\geq 1$). In this case, the cone spanned by partial Hasse invariants $\Ccal_{\pha}$ is no longer the relevant cone that controls cohomology vanishing, and needs to be replaced with the zip cone $\Ccal_{\zip}$. This cone is a group-theoretical object attached to the stack of $G$-zips of Pink--Wedhorn--Ziegler. We now explain our results in more detail. Let $n\geq 1$ be an integer and let $\Acal_{n}$ denote the moduli stack over $\ZZ$ of principally polarized abelian varieties of rank $n$. We may also consider the moduli space $\Acal_{n,K}$ with $K$-level structure, where $K$ is a compact open subgroup of the $\AA_f$-points of a general symplectic group $\mathbf{G}=\GSp_{2n,\QQ}$. Let $p$ be a prime number and assume that $K=K_p K^p$ where $K^p\subset \mathbf{G}(\AA_f^p)$ is compact open and $K_p\subset \mathbf{G}(\QQ_p)$ is hyperspecial (we say that $p$ is a prime of good reduction). In this case, the scheme $\Acal_{n,K}$ extends naturally to a smooth, quasi-projective $\ZZ_{(p)}$-scheme. Denote by $\overline{\Acal}_{n,K}\colonequals \Acal_{n,K} \otimes_{\ZZ_{(p)}} \overline{\FF}_p$ the special fiber over $\overline{\FF}_p$ of $\Acal_{n,K}$. Let $\pi\colon \Ascr\to \Acal_{n,K}$ denote the universal abelian scheme over $\Acal_{n,K}$. The sheaf $\Omega=\pi_*\Omega^1_{\Ascr/\Acal_{n,K}}$ is called the Hodge vector bundle of $\Acal_{n,K}$. For any $n$-tuple $\lambda=(k_1,\dots,k_n)$ of integers we can naturally attach a vector bundle $\Vcal_I(\lambda)$ on $\Acal_{n,K}$, by applying a Schur functor to $\Omega$. In our sign convention, one has $\Omega=\Vcal_I(0,\dots,0,-1)$, and its determinant $\omega=\bigwedge^n\Omega$ (the Hodge line bundle) corresponds to $\lambda=(-1,\dots,-1)$. If we denote by $B_0$ the lower-triangular subgroup of $\GL_{n,\QQ}$ and view $\lambda=(k_1,\dots,k_n)$ as a character of the diagonal torus $T_0\subset B_0$, then $\Vcal_I(\lambda)$ is modeled on the induced $\GL_{n,\QQ}$-representation $V_I(\lambda)=\ind_{B_0}^{\GL_n}(\lambda)$.

For any field $F$ which is a $\ZZ_{(p)}$-algebra, the elements of $H^0(\Acal_{n,K}\otimes_{\ZZ_{(p)}} F,\Vcal_I(\lambda))$ are called Siegel automorphic forms of weight $\lambda$, level $K$, with coefficients in $F$. We are mainly interested in the cases $F=\overline{\FF}_p$. In several previous papers (e.g. \cite{Goldring-Koskivirta-global-sections-compositio, Goldring-Koskivirta-divisibility}), we studied (for general Shimura varieties of Hodge-type) the question of determining which weights $\lambda$ admit nonzero automorphic forms over $F$. We define the following set:
\begin{equation}\label{CKF-intro} \tag{2}
    C_K(F) \colonequals \{\lambda\in \ZZ^n \ | \ H^0(\Acal_{n,K}\otimes_{\ZZ_{(p)}} F, \Vcal_I(\lambda)) \neq 0 \}.
\end{equation}
This set depends on the choice of level $K$. However, the author showed in \cite[Corollary 1.5.3]{Koskivirta-automforms-GZip} that the saturation of this set $\Ccal(F)$ is independent of $K$. Here, the saturation is the set of $\lambda\in X^*(T)$ such that some positive multiple lies in $C_K(F)$. For $F=\CC$, the answer is given by the Griffiths--Schmid conditions. In the case of Siegel modular varieties, in our sign convention, the Griffiths--Schmid cone is defined by
\begin{equation}
\Ccal_{\GS}=\left\{ \lambda=(a_1,\dots, a_n)\in \ZZ^n \ \relmiddle| \ 0\geq a_1\geq \dots \geq a_n \right\}.
\end{equation}
It is known to experts (but references in the literature are scarce) that $\Ccal(\CC)=\Ccal_{\GS}$. We showed that $\Ccal(\CC)\subset \Ccal_{\GS}$ for general Hodge-type Shimura varieties via characteristic $p$ methods in \cite{Goldring-Koskivirta-GS-cone}.

On the other hand, very little is known on the cone $\Ccal(\overline{\FF}_p)$. We conjectured in \cite{Goldring-Koskivirta-global-sections-compositio} that this set is encoded by the stack of $G$-zips. Specifically, write $G$ for the general symplectic group $G=\GSp_{2n,\FF_p}$ over $\FF_p$. Let $\mu\colon \GG_{\mathrm{m}}\to G$ be a minuscule cocharacter. By results of Wedhorn--Viehmann and Zhang, there exists a smooth surjective morphism of stacks
\begin{equation}
  \zeta \colon  \overline{\Acal}_{n,K} \to \GZip^\mu
\end{equation}
where $\GZip^\mu$ is the stack of $G$-zips of type $\mu$. The fibers of $\zeta$ are called the Ekedahl--Oort strata and are determined by the isomorphism classe of the $p$-torsion of the underlying abelian variety. 
We can define the zip cone similarly to \eqref{CKF-intro} as the set of $\lambda\in \ZZ^n$ such that $\Vcal_I(\lambda)$ admits nonzero section on $\GZip^\mu$. This set can be described in terms of representation theory of reductive groups. Specifically, for $\lambda\in \ZZ^n$, write again $V_I(\lambda)=\ind_{B_0}^{\GL_n}(\lambda)$ (viewed now as an algebraic representation of $\GL_{n,\FF_p}$). Let $V_I(\lambda)=\bigoplus_{\chi\in \ZZ^n} V_I(\lambda)_\chi$ be the $T_0$-weight decomposition and define $V_I(\lambda)_{\leq 0}$ as the direct sum of the weight spaces $V_I(\lambda)_\chi$ satisfying $\chi=(a_1,\dots,a_n)$ with $a_n\leq 0$. The author showed (\cite{Koskivirta-automforms-GZip}) that the space $H^0(\GZip^\mu,\Vcal_I(\lambda))$ identifies with the intersection $V_I(\lambda)^{L(\FF_p)} \cap V_I(\lambda)_{\leq 0}$. Therefore, we have:
\begin{equation}\label{Czip-intro}\tag{3}
C_{\zip} = \{\lambda\in \ZZ^n \ | \ V_I(\lambda)^{\GL_n(\FF_p)} \cap V_I(\lambda)_{\leq 0} \neq 0 \}.
\end{equation}
Denote by $\Ccal_{\zip}$ the saturation of $C_{\zip}$. By pullback via the map $\zeta$, we always have $\Ccal_{\zip}\subset \Ccal(\overline{\FF}_p)$. The main result of this paper is the following:
\begin{thm1}
For the Siegel-type Shimura variety $\overline{\Acal}_{n,K}$, we have $\Ccal(\overline{\FF}_p)=\Ccal_{\zip}$.
\end{thm1}
This theorem is a vanishing result for the coherent cohomology in degree $0$ of automorphic vector bundles. Indeed, it states that for all $\lambda\notin \Ccal_{\zip}$, the cohomology group $H^0(\overline{\Acal}_{n,K}, \Vcal_I(m\lambda))$ is zero for any $m\geq 1$. More generally, we conjectured (\cite{Goldring-Koskivirta-global-sections-compositio}) that the equality $\Ccal(\overline{\FF}_p)=\Ccal_{\zip}$ hold for any Hodge-type Shimura variety at a place of good reduction. A few instances of this conjecture were proved in previous papers: The case of Hilbert--Blumenthal Shimura varieties was carried out in \cite{Goldring-Koskivirta-global-sections-compositio} (and independently by Diamond--Kassaei in \cite{Diamond-Kassaei}). We also showed this conjecture for unitary Shimura varieties attached to $\GU(r,s)$ for a totally imaginary quadratic extension $\mathbf{E} / \QQ$ when $r+s\leq 4$ (except when $r=s=2$ and $p$ is inert) in \cite{Goldring-Koskivirta-divisibility}. In the same paper, we treated the case of Siegel modular varieties of rank $2$ and $3$. The author proved the conjecture for unitary Shimura varieties $\U(2)$ (for an arbitrary CM extension $\mathbf{E} / \mathbf{F}$). The verification of the conjecture for the whole Siegel family $(\overline{\Acal}_{n,K})$ is a major achievement in our program.

All these results were proved using a tedious, combinatorial approach based on the notion of intersection-sum cone. Although this approach provides useful information (for example, regarding divisibility of automorphic forms), it is computationally too difficult to be implemented when the Weyl group of the reductive group $G$ is large. Therefore, in this paper we use a different approach and prove the equality $\Ccal(\overline{\FF}_p)=\Ccal_{\zip}$ directly, but without describing this set explicitly. Using \eqref{Czip-intro}, it becomes now a representation-theoretical question to determine the cone $\Ccal(\overline{\FF}_p)$. The explicit form of this cone is investigated in work in progress. For $n=3$, we are able to give explicit equations for $\Ccal(\overline{\FF}_p)$. In particular, we show the following optimal vanishing result:
\begin{thm2}
Let $\lambda=(k_1,k_2,k_3)\in \ZZ^3$ and assume that $p^2 k_1+k_2+pk_3 > 0$ or that $pk_1+p^2 k_2+k_3 > 0$, Then one has
    \begin{equation}
H^0(\overline{\Acal}_{3,K},\Vcal_I(\lambda)) = 0.
    \end{equation}
\end{thm2}
For $p\geq 5$, this result was already shown in \cite{Goldring-Koskivirta-divisibility}. The present article shows that the assumption $p\geq 5$ (which was essential in \loccit) is superfluous, and provides a much more enlightening proof.

We briefly explain the proof of Theorem 1. The theory of $G$-zips shows that there is a natural $\GL_n(\FF_p)$-principal bundle on the ordinary locus $\overline{\Acal}_{n,K}^{\ord}$. This allows us to define a subbundle $\Vcal_I(\lambda)_{\triv}\subset \Vcal_I(\lambda)$ (defined only over $\overline{\Acal}_{n,K}^{\ord}$), corresponding to the sub $\GL_n(\FF_p)$-representation $V_I(\lambda)^{\GL_n(\FF_p)}\subset V_I(\lambda)$ (see section \ref{sec-triv-vb}). We say that an automorphic form $\lambda\in H^0(\overline{\Acal}_{n,K},\Vcal_I(\lambda))$ is of trivial-type if the restriction of $f$ to $\overline{\Acal}_{n,K}^{\ord}$ lies in the subspace
\begin{equation}
H^0(\overline{\Acal}_{n,K}^{\ord},\Vcal_I(\lambda)_{\triv}) \ \subset \ H^0(\overline{\Acal}_{n,K}^{\ord},\Vcal_I(\lambda)).
\end{equation}
Denote by $H^0(\overline{\Acal}_{n,K},\Vcal_I(\lambda))_{\triv}$ the subspace of automorphic forms of weight $\lambda$ of trivial-type. One of our main ingredients is the following:
\begin{thm3}
Assume that $H^0(\overline{\Acal}_{n,K},\Vcal_I(\lambda))\neq 0$. Then there exists $d\geq 1$ such that $H^0(\overline{\Acal}_{n,K},\Vcal_I(d\lambda))_{\triv}$ is nonzero.
\end{thm3}

More precisely, denote by $R$ and $R_{\triv}$ the direct sum over $\lambda\in \ZZ^n$ of $H^0(\overline{\Acal}_{n,K},\Vcal_I(\lambda))$ and $H^0(\overline{\Acal}_{n,K},\Vcal_I(\lambda))_{\triv}$ respectively. Then $R$, $R_{\triv}$ are naturally endowed with a structure of $\overline{\FF}_p$-algebra, and $R_{\triv}\subset R$ is a subalgebra. We prove that this extension is integral (Theorem \ref{thm-Rtriv-int}). Theorem 3 above is an easy consequence of this result. To prove Theorem 1, we show that the space $H^0(\overline{\Acal}_{n,K},\Vcal_I(\lambda))$ exhibits the same pattern as $H^0(\GZip^\mu,\Vcal_I(\lambda))$, i.e. is determined by the interaction of an "ordinary" part $V_I(\lambda)^{\GL_n(\FF_p)}$ and a "negative" part $V_I(\lambda)_{\leq 0}$ as in \eqref{Czip-intro} above. In the case of $\Acal_{n,K}$, the trivial-type automorphic forms correspond to the ordinary part. The negative part is "detected" via an embedding of a Hilbert modular variety $X_{\mathbf{F}}$ for which $p$ splits in $\mathbf{F}$ and corresponds to the condition \eqref{neg-intro}.

\subsection*{Outline of the paper}
In section 1, we recall some general facts about quotient stacks, and introduce the symmetric transforms attached to an $H$-principal bundle, when $H$ is a finite \'{e}tale group scheme. In section 2, we review the stack of $G$-zips, and vector bundles thereon. The following section is devoted to the theory of Shimura varieties, especially Siegel and Hilbert modular varieties, and their connection with the stack of $G$-zips. In section 5, we define the trivial part (over the ordinary locus) of an automorphic vector bundle and prove Theorem 3. Finally, in the last section we explain the proof of our main result Theorem 1.

\subsection*{Acknowledgements}
We are grateful to Wushi Goldring for many helpful conversations on related topics. We also would like to thank Fred Diamond and Payman Kassaei for organizing a workshop at King's College London in March 2024, where the Cone Conjecture was among the topics of discussion. This work was supported by JSPS KAKENHI Grant Number 21K13765.

\section{Vector bundles and quotient stacks}

\subsection{Associated sheaf of a representation} \label{sec-assoc-sheaf}

Fix an algebraically closed field $k$. Let $H$ be an algebraic $k$-group acting on the left on a $k$-scheme $X$. Denote by $[H\backslash X]$ the associate quotient stack. Let $\rho\colon H\to \GL(V)$ be an algebraic representation of $H$ on a finite-dimensional $k$-vector space $V$. By the associated sheaf construction (\cite[I.5.8]{jantzen-representations}), one can attach to $\rho$ a vector bundle $\Vcal(\rho)$ on the stack $[H\backslash X]$. The space of global sections $H^0([H\backslash X],\Vcal(\rho))$ can be identified with the $k$-vector space of regular maps $f\colon X\to V$ satisfying the identity:
\begin{equation}\label{glob-sec}
f(g\cdot x)=\rho(g) f(x), \quad g\in G, \ x\in X.
\end{equation}

\subsection{Torsors and global sections}\label{sec-torsors}
The classifying stack of principal $H$-bundles is the quotient stack $B(H)\colonequals [H\backslash \spec(k)]$. If $X$ is a $k$-scheme endowed with a morphism of stacks $\zeta\colon X\to B(H)$, then the fiber product
\begin{equation}
    \xymatrix@1@M=5pt{
Y \ar[r] \ar[d]_-{\beta} & \spec(k) \ar[d] \\
X \ar[r]_-{\zeta} & B(H)    }
\end{equation}
gives a principal $H$-bundle $\beta\colon Y\to X$. In particular $Y$ is endowed with a left action of $H$ and the quotient stack $[H\backslash Y]$ identifies with $X$. In particular, for any algebraic $H$-representation $\rho\colon H\to \GL(V)$, global sections of the vector bundle $\zeta^*(\Vcal(\rho))$ identify with $H$-equivariant regular maps $f\colon Y\to V$, as in \eqref{glob-sec} above.

\subsection{Finite \'{e}tale group schemes} \label{sec-fin-et}
Assume now that $H$ is a finite \'{e}tale group scheme over $k$ and let $\zeta\colon X\to B(H)$ be a morphism. Let $\rho\colon H\to \GL(V)$ be an algebraic representation. We denote by $V^{H}\subset V$ the $H$-invariant subspace. Clearly, $V^H$ is a sub-$H$-representation of $V$. Write $\Vcal(\rho)_{\triv}\subset \Vcal(\rho)$ for the vector bundle on $B(H)$ associated with $V^H$. Similarly, if $\Vcal\colonequals \zeta^*(\Vcal(\rho))$, we write simply $\Vcal_{\triv}$ for the subvector bundle $\zeta^*(\Vcal(\rho)_{\triv})$. By section \ref{sec-torsors} above, the space $H^0(X,\Vcal_{\triv})$ can  be expressed as follows:
\begin{equation}
    H^0(X,\Vcal_{\triv}) = \left\{ f\colon Y\to V^H \ \textrm{regular} \ \relmiddle| \ f(h\cdot y)=f(y), \  h\in H, y\in Y \right\}.
\end{equation}
Note that in the trivial case when $X=B(H)$, $\zeta=\iden_{B(H)}$, we have (by definition) an identification
\begin{equation}\label{BH-triv}
    H^0(B(H),\Vcal) = H^0(B(H),\Vcal_{\triv})= V^H,
\end{equation}
where we used the obvious relation $(V^H)^H = V^H$.

\subsection{Symmetric transforms}\label{sec-norm-map}
We continue to assume that $H$ is a finite \'{e}tale group scheme over $k$. Let $X$ be a $k$-scheme endowed with a morphism $\zeta\colon X\to B(H)$ and write $\beta\colon Y\to X$ for the attached $H$-principal bundle on $X$. Write $N$ for the cardinality of the finite group $H$. Define the norm and trace maps of $V$ as follows:
\begin{align*}
    &\Norm_H\colon V\to \Sym^N(V), \quad x\mapsto \bigotimes_{h\in H} h\cdot x \\
    &\Tr_H\colon V\to V, \quad x\mapsto \sum_{h\in H} h\cdot x.
\end{align*}
More generally, for any integer $1\leq d\leq N$, we may define the $d$-th symmetric function
\begin{equation}
    \sfr^{(d)}_{H}\colon V\to \Sym^d(V), \quad x\mapsto S_d(\{h\cdot x\}_{h\in H})
\end{equation}
where $S_d(\{X_{h}\}_{h\in H})$ is the $d$-th symmetric polynomial, i.e. the coefficient of degree $N-d$ in the expression $\prod_{h\in H}(t-X_h)$. It is clear that the maps $\sfr^{(d)}_{H}$ are regular morphisms of affine varieties over $k$. Furthermore, by construction the element $\sfr^{(d)}_{H}(x)$ lies in the $H$-invariant part $\left(\Sym^d(V)\right)^H$. Let $f\in H^0(X,\Vcal)$ be a global section of $\Vcal\colonequals \zeta^*(\Vcal(\rho))$. View $f$ as a  regular $H$-equivariant map $f\colon Y\to V$, where $\beta\colon Y\to X$ is the $H$-principal bundle attached to $\zeta$. The composition $\sfr^{(d)}_{H} \circ f \colon Y\to \left(\Sym^d(V)\right)^H$ is regular and $H$-equivariant, hence yields a section
\begin{equation}
    \sfr^{(d)}_{H}(f)\in H^0(X, \Sym^d(\Vcal)_{\triv}).
\end{equation}
This construction gives for each $1\leq d\leq N$ a (non-linear) operator 
\begin{equation}
    \sfr^{(d)}_{H}\colon H^0(X,\Vcal) \to H^0(X, \Sym^d(\Vcal)_{\triv}).
\end{equation}

\section{\texorpdfstring{The stack of $G$-zips}{}}

\subsection{Preliminaries}\label{sec-prelim}

We fix an algebraic closure $k$ of $\FF_p$. For a $k$-scheme $X$, we denote by $X^{(p)}$ the $p$-power Frobenius twist of $X$. The relative Frobenius morphism is denoted by $\varphi_X\colon X \to X^{(p)}$ or simply $\varphi$. Let $G$ be a connected, reductive group over $\FF_p$ and let $\mu\colon \GG_{\mathrm{m},k}\to G_k$ be a cocharacter. We call the pair $(G,\mu)$ a cocharacter datum. The action of $\mu$ induces a decomposition $\Lie(G)=\bigoplus_{n\in \ZZ}\Lie(G)_n$, where $\Lie(G)_n$ is the subspace where $\GG_{\mathrm{m},k}$ acts via $x\mapsto x^n$ through $\mu$. This gives a pair of opposite parabolic subgroups $P_{\pm}\subset G_k$, uniquely characteriszed by the conditions
\begin{equation}\label{Pminplus}
    \Lie(P_-)=\bigoplus_{n\leq 0} \Lie(G)_n \quad \textrm{and} \quad \Lie(P_+)=\bigoplus_{n\geq 0} \Lie(G)_n.
\end{equation}
We set $P\colonequals P_-$ and $Q\colonequals P_+^{(p)}$. Let $L\colonequals \cent(\mu)$ be the centralizer of $\mu$, which is a Levi subgroup of $P$. Put $M\colonequals L^{(p)}$, a Levi subgroup of $Q$. Since $M=L^{(p)}$, we have the Frobenius homomorphism $\varphi\colon L\to M$. We call the tuple 
\[\Zcal_\mu \colonequals (G,P,Q,L,M)\]
the zip datum attached to $(G,\mu)$ (this terminology slightly differs from \cite[Definition 3.6]{Pink-Wedhorn-Ziegler-F-Zips-additional-structure}). If $L$ is defined over $\FF_p$, then $M=L^{(p)}=L$. This will be the case for all groups considered in the main part of this paper. Let $\theta_L^P\colon P\to L$ denote the projection onto the Levi subgroup $L$ modulo the unipotent radical $R_{\mathrm{u}}(P)$. Define $\theta_M^Q\colon Q\to M$ similarly. Put
\begin{equation}\label{eq-Edef}
    E\colonequals\{ (x,y)\in P\times Q \mid \varphi(\theta^P_L(x))=\theta_M^Q(y) \}.
\end{equation}
Let $E$ act on $G_k$ by the rule $(x,y)\cdot g\colonequals xgy^{-1}$. The stack of $G$-zips of type $\mu$ is the quotient stack
\begin{equation}\label{GZip-def-eq}
    \GZip^\mu \colonequals \left[E\backslash G_k\right].
\end{equation}
This stack has a modular interpretation in terms of certain torsors (\cite[Definition 1.4]{Pink-Wedhorn-Ziegler-F-Zips-additional-structure}). It is a smooth stack over $k$ whose underlying topological space is finite. The association $(G,\mu) \mapsto \GZip^\mu$ is functorial in the following sense. Let $(H,\mu_H)$ and $(G,\mu_G)$ be two cocharacter data and let $f\colon H\to G$ be a homomorphism defined over $\FF_p$ satisfying $\mu_G=f_k\circ \mu_H$. Then by \cite[\S2.1, \S2.2]{Goldring-Koskivirta-zip-flags}, $f$ induces a natural morphism of stacks $f_{\zip} \colon \HZip^{\mu_H}\to \GZip^{\mu_G}$ which makes the diagram below commute (where the vertical maps are the natural projections).
\begin{equation}\label{GZip-functor}
\xymatrix@M=7pt{
H \ar[r]^{f} \ar[d] & G \ar[d] \\
\HZip^{\mu_H}\ar[r]_{f_{\zip}} & \GZip^{\mu_G}.
}
\end{equation}

\subsection{Notation} \label{sec-notation}

Let $(G,\mu)$ be a cocharacter datum, and write $\Zcal_\mu=(G,P,Q,L,M)$ for the attached zip datum. We fix some group-theoretical data. 
\begin{bulletlist}
\item For simplicity, we always assume that there exists a Borel pair $(B,T)$ defined over $\FF_p$ satisfying $B\subset P$, $T\subset L$ and such that $\mu$ factors through $T$ (this condition can always be achieved after possibly changing $\mu$ to a conjugate cocharacter). In particular, we have an action of $\Gal(k/\FF_p)$ on $X^*(T)$ and $X_*(T)$. Let $B^+$ denote the opposite Borel to $B$ with respect to $T$ (i.e. the only Borel subgroup such that $B\cap B^+=T$). We write $\sigma\in \Gal(k/\FF_p)$ for the $p$-power Frobenius automorphism $x\mapsto x^p$.
\item Let $W=W(G_k,T)$ be the Weyl group of $G_k$. The group $\Gal(k/\FF_p)$ acts on $W$ and the actions of $\Gal(k/\FF_p)$ and $W$ on $X^*(T)$ and $X_*(T)$ are compatible in a natural sense.
\item Write $\Phi$ for the set of $T$-roots and $\Phi^+$ for the positive roots with respect to $B$ (in our convention, a root $\alpha$ is positive if the corresponding $\alpha$-root group $U_\alpha$ is contained in the opposite Borel $B^+$). Let $\Delta$ denote the set of simple roots.
\item For $\alpha \in \Phi$, let $s_\alpha \in W$ be the corresponding reflection. The system $(W,\{s_\alpha \mid \alpha \in \Delta\})$ is a Coxeter system.
\item Write $\ell  \colon W\to \NN$ for the length function, and $\leq$ for the Bruhat order on $W$. Let $w_0$ denote the longest element of $W$. 
\item Write $I\subset \Delta$ for the set of simple roots of $L$.
\item Let $X^*_{+,I}(T)$ denote the set of $I$-dominant characters, i.e. characters $\lambda\in X^*(T)$ such that $\langle \lambda,\alpha^\vee\rangle \geq 0$ for all $\alpha\in I$.
\item For a subset $K\subset \Delta$, let $W_K$ denote the subgroup of $W$ generated by $\{s_\alpha \mid \alpha \in K\}$. Write $w_{0,K}$ for the longest element in $W_K$.
\item Let ${}^KW$ denote the subset of elements $w\in W$ which have minimal length in the coset $W_K w$. Then ${}^K W$ is a set of representatives of $W_K\backslash W$. The longest element in the set ${}^K W$ is $w_{0,K} w_0$.
\item We set $z=\sigma(w_{0,I})w_0$. The triple $(B,T,z)$ is a $W$-frame, in the terminology of \cite[Definition 2.3.1]{Goldring-Koskivirta-zip-flags}.
\item For $w\in W$, fix a representative $\dot{w}\in N_G(T)$, such that $(w_1w_2)^\cdot = \dot{w}_1\dot{w}_2$ whenever $\ell(w_1 w_2)=\ell(w_1)+\ell(w_2)$ (this is possible by choosing a Chevalley system, \cite[ XXIII, \S6]{SGA3}). If no confusion occurs, we simply write $w$ instead of $\dot{w}$.
\item For $w,w'\in {}^I W$, write $w'\preccurlyeq w$ if there exists $w_1\in W_I$ such that $w'\leq w_1 w \sigma(w_1)^{-1}$. This defines a partial order on ${}^I W$ (\cite[Corollary 6.3]{Pink-Wedhorn-Ziegler-zip-data}).
\end{bulletlist}

\subsection{Parametrization of strata}\label{sec-param}

Using the notation explained above, we can now give the parametrization of $E$-orbits in $G_k$. For $w\in {}^I W$, define $G_w$ as the $E$-orbit of $\dot{w}\dot{z}^{-1}$.

\begin{theorem}[{\cite[Theorem 7.5]{Pink-Wedhorn-Ziegler-zip-data}}] \label{thm-E-orb-param} \ 
\begin{assertionlist}
\item Each $E$-orbit is smooth and locally closed in $G_k$.
\item The map $w\mapsto G_w$ is a bijection from ${}^I W$ onto the set of $E$-orbits in $G_k$.
\item For $w\in {}^I W$, one has $\dim(G_w)= \ell(w)+\dim(P)$.
\item The Zariski closure of $G_w$ is 
\begin{equation}\label{equ-closure-rel}
\overline{G}_w=\bigsqcup_{w'\in {}^IW,\  w'\preccurlyeq w} G_{w'}.
\end{equation}
\end{assertionlist}
\end{theorem}

For each $w\in {}^I W$, we define the quotient stack
\begin{equation}
    \Vcal_w \colonequals [E\backslash G_w].
\end{equation}
It is a locally closed substack of $\GZip^\mu=[E\backslash G_k]$. We call the substacks $\Vcal_w$ the zip strata of $\GZip^\mu$. We obtain a stratification $\GZip^\mu = \bigsqcup_{w\in {}^I W} \Vcal_w$ and the closure relations between strata are given by \eqref{equ-closure-rel}. 

\subsection{The ordinary locus}

In particular, the above description of $E$-orbits of $G_k$ shows that there is a unique open $E$-orbit
$U_\mu\subset G$ corresponding to the longest element $w_{0,I}w_0$ in ${}^I W$. We always have $1\in U_\mu$. The corresponding open substack 
\begin{equation}
\Ucal_\mu\colonequals \Vcal_{w_{0,I}w_0} = \left[E \backslash U_\mu \right]    
\end{equation}
is called the \emph{$\mu$-ordinary locus} of $\GZip^\mu$. This terminology come from the theory of Shimura varieties. Furthermore, when $P$ is defined over $\FF_p$, one simply calls $\Ucal_\mu$ the "ordinary locus" and denote it by $\Ucal^{\ord}$ (we also write $U^{\ord}$ for the open $E$-orbit $U_\mu$). Again, this terminology stems from Shimura varieties, see section \ref{sec-HT-SV} below for details. The stacks of $G$-zips that will appear in this paper satisfy that $P$ is defined over $\FF_p$. In this case, one can show (\cite{Koskivirta-Wedhorn-Hasse}) that the stabilizer in $E$ of the element $1\in G_k$ is of the form below:
\begin{equation}
    \Stab_E(1) = \{(x,x) \ | \ x\in L(\FF_p)\}.
\end{equation}
In particular, this group identifies with the \'{e}tale group scheme $L(\FF_p)$. We deduce that there is an isomorphism $E/L(\FF_p)\simeq U^{\ord}$. Hence, we deduce
\begin{equation}\label{Ucal-ord}
    \Ucal^{\ord} \simeq \left[ E\backslash E / L(\FF_p) \right] \simeq \left[ 1/L(\FF_p) \right] = B(L(\FF_p)).
\end{equation}
In particular, this shows that there exists a universal $L(\FF_p)$-principal bundle over $\Ucal^{\ord}$.

\subsection{\texorpdfstring{Vector bundles on $\GZip^\mu$}{}} \label{sec-vb-Gzip}

Let $\rho\colon P\to \GL(V)$ be an algebraic $P$-representation on a finite-dimensional $k$-vector space. We view $\rho$ as a representation of $E$ via the first projection $\pr_1\colon E\to P$. Since $\GZip^\mu =[E\backslash G_k]$, we can attach to $\rho$ a vector bundle $\Vcal(\rho)$, as explained in section \ref{sec-assoc-sheaf}. In particular, any algebraic $L$-representation can be viewed as a $P$-representation trivial on the unipotent radical of $P$, and hence gives rise to a vector bundle on $\GZip^\mu$. In this paper, we are mostly interested in a family of vector bundles $\Vcal_I(\lambda)$ parametrized by characters $\lambda\in X^*(T)$, that are related to the theory of vector-valued automorphic forms. For $\lambda\in X^*(T)$, view $\lambda$ as a one-dimensional representation of $B$, and define $V_I(\lambda)$ as the induced representation
\begin{equation}
    V_I(\lambda)\colonequals \ind_B^P(\lambda).
\end{equation}
Explicitly, $V_I(\lambda)$ is the $L$-representation whose elements consists of regular maps $f\colon P\to \AA^1$ satisfying the relation
\begin{equation}
    f(xb)=\lambda^{-1}(b)f(x), \quad x\in P, \ b\in B.
\end{equation}
The action of $L$ on $V_I(\lambda)$ is defined as follows: If $f\colon P\to \AA^1$ is an element of $V_I(\lambda)$ and $a\in L$, then $a\cdot f$ is the function $P\to \AA^1$, $x\mapsto f(a^{-1}x)$. The representation $V_I(\lambda)$ has highest weight $\lambda$ and the socle of $V_I(\lambda)$ is the unique irreducible representation of highest weight $\lambda$ over $k$. We are interested in the following set, that will prove to be closely related to automorphic forms in characteristic $p$:
\begin{equation} 
    C_{\zip} \colonequals \{\lambda\in X^*(T) \ | \ H^0(\GZip^\mu, \Vcal_I(\lambda)) \neq 0 \}.
\end{equation}
This is a subcone of $X^*(T)$ (i.e. an additive submonoid containing zero). For a subcone $C\subset X^*(T)$, define the saturation of $C$ as the set
\begin{equation}
    \Ccal\colonequals \{ \lambda\in X^*(T) \ | \ m\lambda\in C \ \textrm{for some } m\geq 1 \}.
\end{equation}
We always denote the saturation using the calligraphic letter $\Ccal$. For example, we denote by $\Ccal_{\zip}$ the saturation of $C_{\zip}$. Since $\Vcal_I(\lambda)=0$ when $\lambda$ is not $I$-dominant, we obviously have an inclusion
\begin{equation}
    \Ccal_\zip\subset X^*_{+,I}(T).
\end{equation}
We will see that the set $\Ccal_{\zip}$ controls the vanishing of cohomology of automorphic vector bundles in characteristic $p$. When $(G,\mu)$ is a cocharacter datum of Hasse-type (\cite[Definition 4.1.6]{Imai-Koskivirta-zip-schubert}), Imai and the author showed that the cone $\Ccal_{\zip}$ coincides with the cone $\Ccal_{\pha}$ spanned by the weights of partial Hasse invariants (\loccitn, Theorem 4.3.1). In general, the exact determination of $\Ccal_{\zip}$ is a difficult problem.

\subsection{\texorpdfstring{Global sections of vector bundles on $\GZip^\mu$}{}} \label{sec-glab-sec-Gzip}

For any algebraic $P$-representation $\rho\colon P\to \GL(V)$, the space of global sections $H^0(\GZip^\mu,\Vcal(\rho))$ was described in \cite{Imai-Koskivirta-vector-bundles} as the invariant part of the Brylinski--Kostant filtration under a certain finite algebraic subgroup of $L$. Since we are mainly interested in the case of automorphic vector bundle on Siegel-type Shimura varieties, we will assume for simplicity that $G$ is split over $\FF_p$ and that $\rho$ is an $L$-representation $\rho\colon L\to \GL(V)$. Under these assumptions, we can easily describe the space $H^0(\GZip^\mu,\Vcal(\rho))$ as follows. First, denote by $V^{L(\FF_p)}$ the subspace of $L(\FF_p)$-invariant vectors of $V$. Since $\Ucal^{\ord}$ identifies with $B(L(\FF_p))$ by \eqref{Ucal-ord}, we deduce that there is an identification
\begin{equation}
    H^0(\Ucal^{\ord},\Vcal(\rho)) = V^{L(\FF_p)}.
\end{equation}
In particular, $H^0(\GZip^\mu,\Vcal(\rho))$ identifies with a subspace of $V^{L(\FF_p)}$. On the other hand, consider the $T$-weight decomposition of $V$:
\begin{equation}
    V=\bigoplus_{\chi\in X^*(T)} V_\chi.
\end{equation}
Define the "non-positive part" of $V$ as the subspace
\begin{equation}\label{nonposV}
    V_{\leq 0}\colonequals \bigoplus_{\substack{\langle \chi,\alpha^\vee\rangle\leq 0 \\ \forall \alpha \in \Delta^P}} V_\chi.
\end{equation}
Then, we have the following result:
\begin{theorem}[{\cite[Theorem 3.7.2]{Koskivirta-automforms-GZip}, \cite[Theorem 3.4.1]{Imai-Koskivirta-vector-bundles}}]\label{thm-H0-zip}
One has an identification
\begin{equation}
    H^0(\GZip^\mu,\Vcal(\rho)) = V^{L(\FF_p)} \cap V_{\leq 0}.
\end{equation}
\end{theorem}
In particular, we can interpret the set $C_{\zip}$ as the set of $\lambda\in X^*(T)$ such that the intersection $V_I(\lambda)^{L(\FF_p)} \cap V_I(\lambda)_{\leq 0}$ is nonzero. This interpretation makes it possible to use tools from representation theory to study the set $\Ccal_{\zip}$. We hope to investigate this approach in future work.

\section{Siegel and Hilbert--Blumenthal Shimura varieties}

\subsection{Siegel modular varieties} \label{Siegel-type-sec}

Let $n\geq 1$ be an integer and let $\Lambda=\ZZ^{2n}$, endowed with the symplectic pairing $\Psi\colon \Lambda\times \Lambda \to \ZZ$ defined by the $2n\times 2n$ symplectic matrix
\begin{equation}
   \left( \begin{matrix}
        & J_0 \\ -J_0 &
    \end{matrix} \right), \quad  \textrm{where} \quad \quad J_0=
    \left( \begin{matrix}
        & &1 \\  & \iddots & \\ 1& &
    \end{matrix} \right) \in \GL_n(\ZZ).
\end{equation}
Let $\GSp_{2n}$ be the reductive group over $\ZZ$ such that for any $\ZZ$-algebra $R$, we have
\begin{equation}
    \GSp_{2n}(R)=\left\{ g\in \GL_R(\Lambda \otimes_\ZZ R) \ \relmiddle| \ \exists c(g)\in R^\times, \ \Psi_R(gx,gy) = c(g) \Psi_R(x,y) \right\}.
\end{equation}
The map $\GSp_{2n}\to \GG_{\mathrm{m}}$, $g\mapsto c(g)$ is the multiplier character. The subgroup defined by the condition $c(g)=1$ is denoted by $\Sp_{2n}$. Let $\Hcal_n^{\pm}$ be the set of symmetric complex matrices of size $n\times n$ whose imaginary part is positive definite or negative definite. The pair $(\GSp_{2n,\QQ},\Hcal_{n}^{\pm})$ is called a Shimura datum of Siegel-type.

We fix a prime number $p$ and set $K_p\colonequals \GSp_{2n}(\ZZ_p)$. It is a hyperspecial subgroup of $\GSp_{2n}(\QQ_p)$. For any open compact subgroup $K^p\subset \GSp_{2n}(\AA^p_f)$, write $K=K_p K^p \subset \GSp_{2n}(\AA_f)$ and let $\Acal_{n,K}$ be the stack over $\ZZ_{(p)}$ such that for any $\ZZ_{(p)}$-scheme $S$, the $S$-valued points of $\Acal_{n,K}$ parametrize the triples $(A,\chi, \overline{\eta})$ satisfying the following conditions:
\begin{bulletlist}
    \item $A$ is an abelian scheme of relative dimension $n$ over $S$.
    \item $\chi\colon A\to A^\vee$ is a $\ZZ_{(p)}$-multiple of a polarization whose degree is prime to $p$.
    \item $\eta\colon \Lambda \otimes_{\ZZ} \AA^p_f \to H^1(A,\AA^p_f)$ is an isomorphism of sheaves of $\AA^p_f$-modules on $S$. We impose that $\eta$ is compatible with the symplectic pairings induced by $\Psi$ and $\chi$. We write $\overline{\eta}=\eta K^p$ for the $K^p$-orbit of $\eta$.
\end{bulletlist}
For $K^p$ small enough, $\Acal_{n,K}$ is a smooth, quasi-projective $\ZZ_{(p)}$-scheme of relative dimension $\frac{n(n+1)}{2}$. We will always make this assumption on $K^p$. We fix an algebraic closure $k=\overline{\FF}_p$ and define:
\begin{equation}
    \overline{\Acal}_{n,K} \colonequals \Acal_{n,K}\otimes_{\ZZ_{(p)}} k.
\end{equation}
We write $\Acal_n$ and $\overline{\Acal}_n$ for the moduli stacks without level structure, parametrizing pairs $(A,\chi)$ as above. We let $\Ascr\to \Acal_{n,K}$ denote the universal abelian scheme over $\Acal_{n,K}$ and $e\colon \Acal_{n,K}\to \Ascr$ the unit section. Define the Hodge bundle of $\Acal_{n,K}$ by
\begin{equation}\label{Omega}
    \Omega\colonequals e^*(\Omega^1_{\Ascr/\Acal_{n,K}}).
\end{equation}
We let $\omega\colonequals \bigwedge^n \Omega$ denote the determinant of $\Omega$ and call it the Hodge line bundle.


\subsection{Hodge-type Shimura varieties}\label{sec-HT-SV}

More generally, let $\gx$ be a Shimura datum of Hodge-type \cite[2.1.1]{Deligne-Shimura-varieties}, where $\mathbf{G}$ is a connected, reductive group over $\QQ$. By definition, there is an embedding of Shimura data $\gx\to (\GSp_{2n,\QQ},\Hcal_n^{\pm})$ for some $n\geq 1$. The symmetric domain $\mathbf{X}$ gives rise to a well-defined $\mathbf{G}(\overline{\QQ})$-conjugacy class of cocharacters $[\mu]$ of $\mathbf{G}_{\overline{\QQ}}$. Write $\mathbf{E}=\mathbf{E}(\mathbf{G},\mathbf{X})$ for the reflex field of $\gx$ (i.e. the field of definition of $[\mu]$). For any open compact subgroup $K \subset \mathbf{G}(\AA_f)$, let $\Sh_{K}(\mathbf{G},\mathbf{X})$ be Deligne's canonical model (\cite{Deligne-Shimura-varieties}) at level $K$ defined over $\mathbf{E}$. When $K$ is small enough, $\Sh_{K}(\mathbf{G},\mathbf{X})$ is a smooth, quasi-projective scheme over $\mathbf{E}$, and its $\CC$-valued points are given by
\begin{equation}
    \Sh_{K}(\mathbf{G},\mathbf{X})(\CC) = \mathbf{G}(\QQ) \backslash \left(\mathbf{G}(\AA_f)/K \times \mathbf{X} \right).
\end{equation}
Let $p$ be a prime of good reduction. By this, we mean that $K$ can be written in the form $K=K_p K^p$ where $K_p\subset \mathbf{G}(\QQ_p)$ is hyperspecial and $K^p\subset \mathbf{G}(\AA_f^p)$ is open compact. We will only consider open compact subgroups of this form. In particular, the group $\mathbf{G}_{\QQ_p}$ is unramified and there exists a reductive $\ZZ_p$-model $\mathbf{G}_{\ZZ_p}$ such that $K_p=\mathbf{G}_{\ZZ_p}(\ZZ_p)$. For any place $\pfr|p$ in $\mathbf{E}$, Kisin (\cite{Kisin-Hodge-Type-Shimura}) and Vasiu (\cite{Vasiu-Preabelian-integral-canonical-models}) constructed a smooth canonical model $\Scal_K$ of $\Sh_{K}(\mathbf{G},\mathbf{X})$ over $\Ocal_{\mathbf{E}_\pfr}$. The embedding $\gx\to (\GSp_{2n,\QQ},\Hcal_n^{\pm})$ induces an embedding $\Sh_K(\mathbf{G},\mathbf{X})\to \Acal_{n,K_0}\otimes_{\ZZ_{(p)}}\mathbf{E}_{\pfr}$ for a compatible choice of open compact subgroups $K^p_0\subset \GSp_{2n}(\AA^p_f)$, $K^p\subset \mathbf{G}(\AA_f^p)$. By \cite[Theorem 2.4.8]{Kisin-Hodge-Type-Shimura}, $\Scal_K$ can be constructed as the normalization of the scheme-theoretical image of $\Sh_K(\mathbf{G},\mathbf{X})$ inside $\Acal_{n,K_0}$. Put
\begin{equation}
    S_{K}\colonequals \Scal_K\otimes_{\Ocal_{\mathbf{E}_\pfr}} k.
\end{equation}
Let $K'^p\subset K^p$ be two compact open subgroups of $\mathbf{G}(\AA_f^p)$. There is a natural projection morphism
\begin{equation}\label{change-level-pi}
    \pi_{K',K}\colon \Scal_{K'}\to \Scal_K
\end{equation}
which is finite \'{e}tale. We call $\pi_{K',K}$ the change-of-level map.

We can choose a representative $\mu\in [\mu]$ defined over $\mathbf{E}_\pfr$ by \cite[(1.1.3) Lemma (a)]{Kottwitz-Shimura-twisted-orbital}. 
We can also assume that $\mu$ extends to a cocharacter $\mu \colon \GG_{\mathrm{m},\Ocal_{\mathbf{E}_\pfr}}\to \mathbf{G}_{\ZZ_p}\otimes_{\ZZ_p} \Ocal_{\mathbf{E}_\pfr}$ (\cite[Corollary 3.3.11]{Kim-Rapoport-Zink-uniformization}). Denote by $\mathbf{L} \subset \mathbf{G}_{\mathbf{E}_\pfr}$ the centralizer of the cocharacter $\mu$. Denote by $\mathbf{P}$ the parabolic subgroup of $\mathbf{G}_{\mathbf{E}_\pfr}$ attached to $\mu$, as explained in section \ref{sec-prelim}. Then $\mathbf{L}$ is a Levi subgroup of $\mathbf{P}$. Since $\mathbf{G}_{\QQ_p}$ is unramified, it is quasi-split, so admits a Borel subgroup. We assume that there is a Borel subgroup $\mathbf{B}\subset \mathbf{G}_{\QQ_p}$ contained in $\mathbf{P}$ and a maximal torus $\mathbf{T}\subset \mathbf{B}\cap \mathbf{L}$ such that $\mu$ factors througs $T$. This can always be achieved after changing $\mu$ to a conjugate cocharacter (see \cite[\S 2.5]{Imai-Koskivirta-vector-bundles}).

By properness of the scheme of parabolic subgroups of $\mathbf{G}_{\ZZ_p}$ (\cite[Expos\'{e} XXVI, Corollaire 3.5]{SGA3}), the subgroups $\mathbf{B}$ and $\mathbf{P}$ extend uniquely to subgroups $\mathbf{B}_{\ZZ_p} \subset \mathbf{G}_{\ZZ_p}$ and $\mathbf{P}_{\Ocal_{\mathbf{E}_\pfr}} \subset \mathbf{G}_{\ZZ_p}\otimes_{\ZZ_p} \Ocal_{\mathbf{E}_\pfr}$ respectively. Similarly, let 
$\mathbf{T}_{\ZZ_p}$ be the unique extension of $\mathbf{T}$. Set $G = \mathbf{G}_{\ZZ_p} \otimes_{\ZZ_p} \FF_p$ and denote by $B, T, P$ the special fiber of $\mathbf{B}_{\ZZ_p}, \mathbf{T}_{\ZZ_p}, \mathbf{P}_{\Ocal_{\mathbf{E}_\pfr}}$ respectively. By slight abuse of notation, we denote again by $\mu$ the mod $p$ reduction $\mu \colon \GG_{\mathrm{m},k}\to G_k$. Write $L\subset G_{k}$ for its centralizer. By construction $\Sh_K(\mathbf{G},\mathbf{P})$ is endowed with a natural $\mathbf{P}$-torsor $\Ical_{\mathbf{P}}$ that extends to a $\mathbf{P}_{\Ocal_{\mathbf{E}_\pfr}}$-torsor over $\Scal_K$. Write $I_P$ for the $P$-torsor induced on the special fiber $S_K$.

\subsection{Hilbert--Blumenthal Shimura varieties} \label{hb-var-sec}
Let $\mathbf{F}/\QQ$ be a totally real extension of degree $n$. Write $\mathbf{H}$ for the reductive $\QQ$-group whose $R$-points (for any $\QQ$-algebra $R$) are given by
\begin{equation}
    \mathbf{H}(R)=\left\{ g\in \GL_2(\mathbf{F}\otimes_\QQ R) \ \relmiddle| \ \det(g)\in R^\times \right\}.
\end{equation}
It is a subgroup of the Weil restriction $\Res_{\mathbf{F}/\QQ}(\GL_{2,\mathbf{F}})$. We assume that $p$ is unramified in $\mathbf{F}$. In this case, the group $\mathbf{H}$ is unramified at $p$ and the lattice $\Ocal_{\mathbf{F}}\otimes_{\ZZ}\ZZ_p \subset \mathbf{F}\otimes_{\QQ}\QQ_p$ gives rise to a reductive $\ZZ_p$-model $\mathbf{H}_{\ZZ_p}$ of $\mathbf{H}_{\QQ_p}$. We set $K'_p\colonequals \mathbf{H}_{\ZZ_p}(\ZZ_p)$, which is a hyperspecial subgroup of $\mathbf{H}(\QQ_p)$. For any open compact subgroup $K'^p \subset \mathbf{H}(\AA_f^p)$, write $K'\colonequals K'_p K'^{p}$. The Hilbert--Blumenthal Shimura variety $\Xcal_{\mathbf{F},K'}$ of level $K'$ is the $\ZZ_{(p)}$-stack whose $S$-valued points (for any $\ZZ_{(p)}$-scheme $S$) parametrize quadruples $(A,\chi, \iota, \overline{\eta})$ as follows:
\begin{bulletlist}
    \item $A$ is an abelian scheme of relative dimension $n$ over $S$.
    \item $\chi\colon A\to A^\vee$ is a $\ZZ_{(p)}$-multiple of a polarization of degree prime to $p$.
    \item $\iota\colon \Ocal_\mathbf{F} \otimes_\ZZ \ZZ_{(p)} \to \End(A)\otimes_\ZZ \ZZ_{(p)}$ is a ring homomorphism.
    \item $\eta\colon (\mathbf{F} \otimes_{\QQ} \AA^p_f)^2 \to H^1(A,\AA^p_f)$ is an $\mathbf{F}$-linear isomorphism of sheaves of $\AA^p_f$-modules on $S$, and $\overline{\eta}=\eta K'^p$ is the $K'^p$-orbit of $\eta$.
\end{bulletlist}
The homomorphism $\iota$ yields an action of $\Ocal_{\mathbf{F}}$ on the dual abelian variety $A^\vee$. We impose that the polarization $\chi$ be $\Ocal_\mathbf{F}$-linear for this action.

Next, we construct a morphism $\Xcal_{\mathbf{F},K'}\to \Acal_{n,K}$ for an appropriate choice of level structures $K'^p\subset \mathbf{H}(\AA_f^p)$ and $K^p\subset \GSp(\AA_f^p)$. Write $\mathbf{G}\colonequals \GSp_{2n,\QQ}$ and $\mathbf{G}_{\ZZ_p}=\GSp_{2n,\ZZ_p}$. First, we construct an embedding $\mathbf{H} \to \mathbf{G}$ of reductive $\QQ$-groups. Consider the symplectic form
\begin{equation}
    \Psi_{0}\colon \mathbf{F}^2\times \mathbf{F}^2\to \QQ, \quad ((x_1,y_1),(x_2,y_2)) \ \mapsto \ \Tr_{\mathbf{F}/\QQ}(x_1y_2-x_2y_1).
\end{equation}
For any matrix $A\in \GL_2(\mathbf{F})$ such that $\det(A)\in \QQ^\times$ and $x,y\in \mathbf{F}^2$, one has
\begin{equation}
    \Psi_0(Ax,Ay)=\det(A)\Psi_0(x,y),
\end{equation}
which shows that $\mathbf{H}\subset \GSp(\Psi_0)$. Fix an isomorphism $\gamma\colon (\mathbf{F}^2,\Psi_0)\to (\Lambda\otimes_{\ZZ}\QQ, \Psi)$ of symplectic spaces over $\QQ$. We obtain an embedding of reductive $\QQ$-groups
\begin{equation}\label{u-embedQ}
u\colon \mathbf{H} \to \mathbf{G}, \quad f\mapsto \gamma \circ f \circ \gamma^{-1}.
\end{equation}
Since $p$ is unramified in $\mathbf{F}$, we may further assume that $\gamma$ induces an isomorphism $\Ocal_{\mathbf{F}}\otimes_{\ZZ}\ZZ_p\to \Lambda\otimes_{\ZZ} \ZZ_p$. Then, $u_{\QQ_p}$ extends to an embedding $u_{\ZZ_p}\colon \mathbf{H}_{\ZZ_p} \to \mathbf{G}_{\ZZ_p}$ of reductive $\ZZ_{p}$-groups. For any compact open subgroups $K'^p\subset \mathbf{H}(\AA_f^p)$ and $K^p\subset \mathbf{G}(\AA_f^p)$ such that $u(K'^p)\subset K^p$, there is a natural morphism of $\ZZ_{(p)}$-schemes
\begin{equation}\label{embed-HB-S}
\widetilde{u}_{K',K} \colon \Xcal_{\mathbf{F},K'} \to \Acal_{n,K}
\end{equation}
defined as follows. Let $x=(A,\chi,\iota,\overline{\eta})$ be a point of $\Xcal_{\mathbf{F},K'}$, where $\overline{\eta}=\eta K'^p$ and $\eta\colon (\mathbf{F} \otimes_{\QQ} \AA^p_f)^2 \to H^1(A,\AA^p_f)$ is an $\mathbf{F}$-linear isomorphism. Then $\widetilde{u}_{K',K}$ sends $x$ to the point $(A,\chi, (\eta\circ \gamma^{-1})K^p)$ of $\Acal_{n,K}$ (the isomorphism $\eta\circ \gamma^{-1}\colon \Lambda\otimes_{\ZZ} \AA^p_f \to H^1(A,\AA^p_f)$ is compatible with the symplectic forms, and thus gives rise to a level structure). Furthermore, we may choose $K', K$ so that $\widetilde{u}_{K',K}$ is a closed embedding, but this will not be necessary for our purpose.

\subsection{\texorpdfstring{$G$-zips and Shimura varieties of Hodge-type}{}}\label{sec-GZip-HT}

We return to the setting of section \ref{sec-HT-SV}, where $S_K=\Scal_K\otimes_{\Ocal_{\mathbf{E}_\pfr}} k$ is the special fiber of a Hodge-type Shimura variety at a place of good reduction. The cocharacter datum $(G,\mu)$ gives rise to a zip datum $\Zcal_\mu=(G,P,Q,L,M)$. Zhang (\cite[4.1]{Zhang-EO-Hodge}) constructed a smooth morphism
\begin{equation}\label{zeta-Shimura}
\zeta_K \colon S_{K}\to \GZip^\mu,
\end{equation}
whose fibers are called the Ekedahl--Oort strata of $S_K$. This map is also surjective by \cite[Corollary 3.5.3(1)]{Shen-Yu-Zhang-EKOR}. The map $\zeta_K$ commutes with change-of-level maps, in the sense that for any compact open subgroups $K'^p\subset K^p \subset \mathbf{G}(\AA_f^p)$, there is a commutative diagram
\begin{equation}
\xymatrix@M=5pt{
     S_{K'} \ar[rd]^{\zeta_{K'}} \ar[dd]_{\pi_{K',K}} & \\
     & \GZip^\mu \\
     S_{K} \ar[ru]_{\zeta_K}
     }
\end{equation}
We often omit the subscript $K$ and simply denote these maps by $\zeta$. For $w\in {}^I W$, write
\begin{equation}
    S_{K,w}\colonequals \zeta^{-1}(\Vcal_{w})
\end{equation}
for the corresponding Ekedahl--Oort stratum. The Ekedahl--Oort stratum $S_{K,w_{0,I}w_0}=\zeta^{-1}(\Ucal_\mu)$ corresponding to the maximal element $w_{0,I}w_0\in {}^I W$ is called the $\mu$-ordinary locus. When $\mathbf{E}_\pfr=\QQ_p$, this is simply the classical ordinary locus of $S_K$, we denote it by $S_K^{\ord}$. It is the set of points $x\in S_K$ whose image in $\Ascr_{n,K_0}\otimes_{\ZZ_{(p)}} k$ corresponds to an ordinary abelian variety. $S^{\ord}_{K}$ is nonempty if and only if $\mathbf{E}_{\pfr}=\QQ_p$.

\subsubsection*{The Siegel case}\label{siegel-zip-sec}
We consider the case $G=\GSp_{2n,\FF_p}$ (explained in section \ref{Siegel-type-sec}). We endow $G$ with the cocharacter
\begin{equation}
    \mu_G\colon \GG_{\mathrm{m},k} \to G_k, \quad t\mapsto \left(
\begin{matrix}
    t I_n & \\ & I_n
\end{matrix}
    \right).
\end{equation}
Let $(u_1,\dots, u_{2n})$ denote the canonical basis of $V = k^{2n}$, and set \[V_P\colonequals \Span_k(u_{n+1},\dots, u_{2n}).\]
Let $P$ be the parabolic subgroup which stabilizes the filtration $0\subset V_P \subset V$. Similarly, let $Q$ be the parabolic subgroup opposite to $P$ stabilizing the subspace $V_Q\colonequals \Span_k(u_1,\dots, u_n)$. The intersection $L\colonequals P\cap Q$ is a common Levi subgroup to $P$ and $Q$. The zip datum attached to $(G,\mu)$ is $\Zcal_\mu = (G,P,Q,L,L)$. Let $E\subset P\times Q$ be the zip group defined in \eqref{eq-Edef} and let $\GZip^{\mu_G}\colonequals [E\backslash G_k]$ be the attached stack of $G$-zips. Let $B\subset G$ be the Borel subgroup of lower-triangular matrices in $G$, and $T\subset B$ the maximal torus consisting of diagonal matrices in $G$. The Weyl group $W=W(G,T)$ is isomorphic to $(\ZZ/2\ZZ)^n \rtimes \Sfr_n$ and can be identified with the group of permutations $w\in \Sfr_{2n}$ satisfying 
\begin{equation}\label{SpW-eq}
    w(i)+w(2n+1-i)=2n+1.
\end{equation}
The Weyl group $W_L=W(L,T)$ is isomorphic to $\Sfr_n$ and identifies with the subgroup of permutations in $\Sfr_{2n}$ satisfying \eqref{SpW-eq} and stabilizing the subset $\{1,\dots, n\}$. 
The torus $T$ is generated by the subgroup $Z\subset T$ of nonzero scalar matrices and the subtorus $T_0\colonequals T\cap \Sp_{2n}$. Explicitly, $T_0$ is given by
\begin{equation}
    T_0 = \{ \diag(t_1,\dots, t_n,t_n^{-1},\dots,t_1^{-1}) \ | \ t_i\in \GG_{\mathrm{m}} \}.
\end{equation}
Identify $X^*(T_0)=\ZZ^n$ such that $(a_1,\dots,a_n)$ corresponds to the character
\begin{equation}\label{charT-Sp}
    \diag(t_1,\dots, t_n,t_n^{-1},\dots,t_1^{-1})\mapsto \prod_{i=1}^n t_i^{a_i}.
\end{equation}
We also identify $X^*(Z)=\ZZ$, such that $b\in \ZZ$ corresponds to the character $\diag(t,\dots,t)\mapsto t^b$. The natural map $X^*(T)\to X^*(T_0)\times X^*(Z)=\ZZ^n\times \ZZ$, $\chi\mapsto (\chi|_{T_0},\chi|_Z)$ is injective, and identifies $X^*(T)$ with the set
\begin{equation}\label{XT-GSp}
    X^*(T) = \left\{(a_1,\dots,a_n , b)\in \ZZ^{n+1} \ \relmiddle| \ \sum_{i=1}^n a_i \equiv b \pmod 2 \right\}.
\end{equation}
By section \ref{sec-GZip-HT}, we have a smooth surjective morphism
\begin{equation}
    \zeta_G\colon \overline{\Acal}_{n,K}\to \GZip^{\mu_G}
\end{equation}
whose fibers are the Ekedahl--Oort strata of $\overline{\Acal}_{n,K}$.

It was shown by Ekedahl--van der Geer that the closure of any Ekedahl--Oort stratum of $\overline{\Acal}_n$ which is not contained in the supersingular locus is irreducible in $\overline{\Acal}_n$ (\cite[Theorem 11.5]{Ekedahl-Geer-EO}). Since $\overline{\Acal}_{n,K}$ may be disconnected, this statement does not hold for $\overline{\Acal}_{n,K}$. However, considering the natural map $\overline{\Acal}_{n,K}\to \overline{\Acal}_{n}$, we see that it holds if we restrict ourselves to a connected component $S\subset \overline{\Acal}_{n,K}$. Write $S^{\ord}$ for the ordinary locus of $S$. In particular, we have the following:
\begin{proposition}\label{prop-non-ord}
Fix a connected component $S\subset \overline{\Acal}_{n,K}$ and assume that $n\geq 2$. Then the non-ordinary locus $S\setminus S^{\ord}$ is irreducible.
\end{proposition}
\begin{proof}
The assumption $n\geq 2$ ensures that $S\setminus S^{\ord}$ is the closure of an Ekedahl--Oort stratum that is not contained in the supersingular locus. The result follows.
\end{proof}

The local ring at the irreducible divisor $S\setminus S^{\ord}$ is a discrete valuation ring. The classical Hasse invariant is a section
\begin{equation}
    \Ha\in H^0(\overline{\Acal}_{n,K},\omega^{p-1})
\end{equation}
whose non-vanishing locus is exactly the ordinary locus $\overline{\Acal}_{n,K}^{\ord}$. Furthermore, $\Ha$ has multiplicity $1$ along the non-ordinary locus. Therefore, if we denote again by $\Ha$ the restriction of the Hasse invariant to $S$, then $\Ha$ is a uniformizer of the local ring at the irreducible divisor $S\setminus S^{\ord}$.

In section \ref{sec-vb-Gzip}, we defined vector bundles $\Vcal_I(\lambda)$ on $\GZip^\mu$ for $\lambda\in X^*(T)$. Via the map $\zeta$, the Hodge vector bundle $\Omega$ (see \eqref{Omega}) and the Hodge line bundle $\omega$ correspond respectively to the bundles $\Vcal_I(\lambda_\Omega)$ and $\Vcal_I(\lambda_\omega)$, where
\begin{align*}
    \lambda_\Omega = (0,\dots, 0 ,-1, 1) \\
   \lambda_\omega = (-1,\dots ,-1, n).
\end{align*}

\subsubsection*{The Hilbert--Blumenthal case}
We simply write $H$ for the special fiber of the reductive group $\textbf{H}_{\ZZ_p}$ defined in section \ref{hb-var-sec}. The embedding $u_{\ZZ_p}\colon \mathbf{H}_{\ZZ_p}\to \mathbf{G}_{\ZZ_p}$ induces an embedding 
\begin{equation}
    u\colon H\to G
\end{equation}
(that we continue to denote by $u$). Let $\pfr_1,\dots,\pfr_r$ be the prime ideals of $\Ocal_{\mathbf{F}}$ above $p$ and write $\kappa_i\colonequals \Ocal_{\mathbf{F}}/\pfr_i$. We may view $H$ as a subgroup
\begin{equation}
    H\subset \prod_{i=1}^r \Res_{\kappa_i/\FF_p} \GL_{2,\kappa_i}.
\end{equation}
The group $H_k$ identifies with
\begin{equation} \label{Hk-ident}
    H_k =\left\{(A_1,\dots, A_n)\in \GL_{2,k}^n \ \relmiddle| \ \det(A_1)=\dots=\det(A_n) \right\}.
\end{equation}
Define a cocharacter of $H$ by:
\begin{equation}
    \mu_H\colon \GG_{\mathrm{m},k}\to H_k, \quad t\mapsto \left(
\left( \begin{matrix}
    t & \\ & 1
\end{matrix}
    \right), \dots , \left( \begin{matrix}
    t & \\ & 1
\end{matrix}
    \right)
    \right).
\end{equation}
Write $B_H$ (resp. $B_H^{+}$) for the subgroup consisting of tuples $(A_1,\dots,A_n)\in H_k$ of lower-triangular (resp. upper-triangular) matrices. It is clear that $B_H$, $B_H^+$ are Borel subgroups of $H$ defined over $\FF_p$. Write $T_H=B_H\cap B_H^+$ for the maximal torus consisting of tuples of diagonal matrices in $H_k$. The zip datum $\Zcal_{\mu_H}$ attached to $\mu_H$ is $\Zcal_{\mu_H}=(H,B_H,B^+_H, T_H,T_H)$. Let $\HZip^{\mu_H}$ be the stack of $H$-zips attached to $(H,\mu_H)$. Write 
\[\zeta_H\colon \overline{\Xcal}_{\mathbf{F},K'}\to \HZip^{\mu_H}\]
for the map given by \eqref{zeta-Shimura}. For a fixed choice of compatible level structures $K'\subset \mathbf{H}(\AA_f^p)$ and $K\subset \mathbf{G}(\AA_f^p)$, we sometimes write $X_H\colonequals \overline{\Xcal}_{\mathbf{F},K'}$ and $X_G\colonequals \overline{\Acal}_{n,K}$ in order to uniformize the notation. Write simply $\widetilde{u}$ for the map $\widetilde{u}_{K',K}\colon \overline{\Xcal}_{\mathbf{F},K'} \to \overline{\Acal}_{n,K}$ given by \eqref{embed-HB-S}.

Note that the embedding $u \colon H\to G$ is compatible with the cocharacters $\mu_H$ and $\mu_G$. Therefore, $u$ induces by functoriality (see diagram \eqref{GZip-functor}) a morphism of stacks
\begin{equation}
    u_{\zip}\colon \HZip^{\mu_H} \to \GZip^{\mu_G}.
\end{equation}
For simplicity, we omit the cocharacters in the notation and write respectively $\HZip$ and $\GZip$ for the above stacks. By the construction of the maps $\zeta_H$ and $\zeta_G$, there is a commutative diagram:
\begin{equation}\label{diag-com-Shim}
\xymatrix@M=8pt{ 
X_H \ar[d]_{\zeta_H} \ar[r]^{\widetilde{u}} & X_G \ar[d]^{\zeta_G} \\
\HZip \ar[r]_{u_{\zip}} & \GZip.
}
\end{equation}
In the proof of our main result, we will use an embedding $\widetilde{u} \colon X_H\to X_G$ for a totally real field $\mathbf{F}$ where $p$ splits completely. In this case, we write down explicitly the embedding $u$ below. Since $p$ is completely split, we have an identification 
\begin{equation}
\Ocal_{\mathbf{F}}\otimes_{\ZZ} \FF_p \simeq \prod_{i=1}^n \Ocal_{\mathbf{F}}/\pfr_i \simeq \FF_p^n.  
\end{equation}
where $\pfr_1,\dots,\pfr_n$ are the prime ideals of $\mathbf{F}$ dividing $p$. We obtain an isomorphism
\begin{equation}
    H\simeq \{(A_1,\dots,A_n)\in \GL_{2,\FF_p} \ | \ \det(A_1)=\dots=\det(A_n)\}.
\end{equation}
Then, the embedding $u\colon H\to G$ identifies with the following map.
\begin{equation}\label{emb-split}
    u \colon \left(
\left(
\begin{matrix}
    a_1&b_1\\ c_1&d_1
\end{matrix}
\right),
\cdots , 
\left(
\begin{matrix}
    a_n&b_n\\ c_n&d_n
\end{matrix}
\right)
    \right) \mapsto 
    \left(
\begin{matrix}
a_1 & & & & & & & b_1 \\
&a_2 & & & & & b_2&  \\
& & \ddots& & &\iddots & & \\
& & & a_n &b_n & & & \\
& & & c_n& d_n& & & \\
& &\iddots & & &\ddots & & \\
& c_2& & & & &d_2 & \\
c_1 & & & & & & &d_1 \\
\end{matrix}
    \right).
\end{equation}
Note that $u$ induces an isomorphism $u\colon T_H\to T$ and a homomorphism $u\colon B_H\to B$. For a character $\lambda\in X^*(T_H)$, denote by $\Lcal(\lambda)$ the induced line bundle on the stack $\HZip^{\mu_H}$, as well as its pullback to $X_H$ via $\zeta_H$. We use this specific notation to avoid confusion with the vector bundles $\Vcal_I(\lambda)$ defined on $\GZip^{\mu_G}$ and their pullbacks to $X_G$.

\subsection{Toroidal compactification} \label{sec-tor}
We return to the setting of a general Shimura variety of Hodge-type (section \ref{sec-HT-SV}). By \cite[Theorem 1]{Madapusi-Hodge-Tor}, there is a sufficiently fine cone decomposition $\Sigma$ and a toroidal compactification $\Scal_K^\Sigma$ of $\Scal_K$ over $\Ocal_{\mathbf{E}_\pfr}$. Fix such a toroidal compactification, and denote by $S_K^\Sigma$ its special fiber. By \cite[Theorem 6.2.1]{Goldring-Koskivirta-Strata-Hasse}, the map $\zeta\colon S_K\to \GZip^\mu$ extends to a map
\begin{equation}
    \zeta^\Sigma\colon S_K^{\Sigma}\to \GZip^\mu.
\end{equation}
Furthermore, by \cite[Theorem 1.2]{Andreatta-modp-period-maps}, the map $\zeta^\Sigma$ is smooth. Since $\zeta$ is surjective, $\zeta^\Sigma$ is also surjective. For $\lambda\in X^*(T)$, denote by $\Vcal^{\Sigma}_I(\lambda)$ the vector bundle $\zeta^{\Sigma,*}(\Vcal_I(\lambda))$. By construction, $\Vcal^{\Sigma}_I(\lambda)$ coincides with the canonical extension of $\Vcal_I(\lambda)$ to $S_K^{\Sigma}$. We have the following Koecher principle:

\begin{theorem}[{\cite[Theorem 2.5.11]{Lan-Stroh-stratifications-compactifications}}] \label{thm-koecher}
The natural map
\begin{equation}
    H^0(S^{\Sigma}_K,\Vcal^{\Sigma}_I(\lambda)) \to H^0(S_K,\Vcal_I(\lambda)) 
\end{equation}
is an isomorphism, except when $\dim(S_K)=1$ and $S^{\Sigma}_K \setminus S_K\neq \emptyset$.
\end{theorem}

\subsection{The flag space}

We defined the stack of $G$-zip flags in \cite{Goldring-Koskivirta-Strata-Hasse} based on previous work by Ekedahl--van der Geer for Siegel modular varieties (\cite{Ekedahl-Geer-EO}). Let $(G,\mu)$ be a cocharacter data, with attached zip datum $\Zcal_\mu$. Fix a Borel pair $(B,T)$ as in \ref{sec-notation}. Define the stack of $G$-zip flags by
\begin{equation}
    \GF^\mu = [E\backslash \left(G_k \times P/B \right)],
\end{equation}
where $E$ acts on $G_k\times P/B$ by the formula $(x,y)\cdot (g,hB)=(xgy^{-1},xhB)$ for all $(x,y)\in E$ and $(g,hB)\in G_k \times P/B$. Similarly to the stack of $G$-zips, this stack has a moduli interpretation in terms of torsors (\cite[Definition 2.1.1]{Goldring-Koskivirta-Strata-Hasse}). The first projection $\pr_1\colon G_k\times P/B\to G_k$ is $E$-equivariant and induces a proper, smooth morphism
\begin{equation}
    \pi\colon \GF^\mu \to \GZip^\mu
\end{equation}
whose fibers are isomorphic to $P/B$. We now let $S_K$ be the good reduction special fiber of a Hodge-type Shimura varieties as in section \ref{sec-HT-SV} endowed with the map $\zeta\colon S_K\to \GZip^\mu$ from \eqref{zeta-Shimura}. We define the flag space of $S_K$ as the fiber product
\begin{equation}\label{diag-flag}
\xymatrix@1@M=5pt{
\Flag(S_K)\ar[r]^-{\zeta_{\flag}} \ar[d]_{\pi_K} & \GF^\mu \ar[d]^{\pi} \\
S_K \ar[r]_-{\zeta} & \GZip^\mu.
}    
\end{equation}
To show irreducibility and connectedness of certain schemes, we will use the lemma below repeatedly in what follows:

\begin{lemma}\label{lem-irred-conn}
Let $k$ be an algebraically closed field, and $f\colon X\to Y$ a morphism of $k$-varieties, with irreducible fibers.
\begin{assertionlist}
    \item If $f$ is proper and $Y$ is connected, then $X$ is connected.
    \item If $f$ is smooth and $Y$ is irreducible, then $X$ is irreducible.
\end{assertionlist}
\end{lemma}

\begin{proof}
We first prove (1). Assume that $X=X_1\sqcup X_2$ for two open closed subschemes $X_1, X_2$. Since the fibers are connected, we have $Y=f(X_1)\sqcup f(X_2)$. By properness $f(X_1)$, $f(X_2)$ are closed. Therefore $f(X_1)=\emptyset$ or $f(X_2)=\emptyset$, and hence either $X_1$ or $X_2$ is empty, which proves (1). For the second assertion, suppose $X=X_1\cup X_2$ for two closed subschemes $X_1, X_2$. For $i=1,2$, let $Y_i$ denote the set of elements such that $f^{-1}(y)\subset X_i$. Since the fibers are irreducible, we have $Y=Y_1\cup Y_2$. Furthermore, we have $Y_i=Y\setminus f(X\setminus X_i)$. Since $f$ is open, $Y_i$ is closed. We deduce $Y_1=Y$ or $Y_2=Y$, and thus $X_1=X$ or $X_2=X$.
\end{proof}

Assume that the ordinary locus $S_K^{\ord}$ of $S_K$ is nonempty. Let $\Flag(S_K)^{\ord}$ denote the preimage of $S_K^{\ord}$ under the map $\pi_K\colon \Flag(S_K)\to S_K$. Similarly, if $S\subset S_K$ denotes a connected component, we write $\Flag(S)^{\ord}$ for the preimage of $S^{\ord}$. In the case $S_K=\overline{\Acal}_{n,K}$, we deduce from Proposition \ref{prop-non-ord} the following:

\begin{corollary}\label{flag-non-ord-irred}
Let $S\subset \overline{\Acal}_{n,K}$ be a connected component. For $n\geq 2$, the non-ordinary locus $\Flag(S)\setminus \Flag(S)^{\ord}$ is irreducible.
\end{corollary}

\subsection{Hecke operators}
Let $\Scal_K$ be the integral model over $\Ocal_{\mathbf{E}_\pfr}$ of a Hodge-type Shimura variety, as in section \ref{sec-HT-SV}. We review the definition of Hecke operators on $\Scal_{K}$. For any $g\in \mathbf{G}(\AA_f^p)$, consider the compact open subgroup $K'=K\cap gKg^{-1}$. There are two natural maps $\Scal_{K\cap g K g^{-1}}\to \Scal_{K}$. The first one is the natural change-of-level map $\pi_{K',K}$ defined in \eqref{change-level-pi}. The second one is induced on $\CC$-points by the map
\begin{equation}
    \mathbf{G}(\AA_f)/K' \to \mathbf{G}(\AA_f)/K, \quad x K' \mapsto gx K.
\end{equation}
We denote the induced map simply by $g\colon \Scal_{K\cap g K g^{-1}}\to \Scal_K$. This construction yields a correspondence with finite \'{e}tale maps

\begin{equation}\label{hecke-diag}
    \xymatrix{
& \Scal_{K\cap g K g^{-1}} \ar[dl]_{\pi}  \ar[dr]^{g} & \\
\Scal_{K} &  & \Scal_{K}
}
\end{equation}
The Hecke algebra $\TT_{K^p}(\mathbf{G})$ (with coefficients in $\ZZ$) attached to $K^p$ is the ring of $K^p$-bi-invariant, locally constant functions $\mathbf{G}(\AA_f^p)\to \ZZ$ with compact support. Any such function can be written as a sum of characteristic functions of double cosets $K^p g K^p$ ($g\in \mathbf{G}(\AA_f^p)$). Multiplication is defined by convolution with respect to the left-Haar measure $\nu$ normalized by $\nu(K^p)=1$. The algebra $\TT_{K^p}(\mathbf{G})$ acts on several natural objects related to $\Scal_{K}$ via the correspondence \eqref{hecke-diag}. For example, it acts on the 0-cycles on the special fiber $S_K$ as follows: for $x\in S_{K}(k)$ and $g\in \mathbf{G}(\AA_f^p)$, we set
\begin{equation}\label{TKp}
    T_{K^p,g}(x) \colonequals g_{*} \pi^*(x).
\end{equation}
We view $T_{K^p,g}(x)$ as a formal sum of points of $S_{K}(k)$. Denote by $\Hcal^p(x)$ the set of all points of $S_{K}(k)$ appearing in the support of a 0-cycle of the form $T_{K^p,g}(x)$ for $g\in \mathbf{G}(\AA_f^p)$. We call $\Hcal^p(x)$ the prime-to-$p$ Hecke orbit of $x$. Van Hoften showed the following density result, as a generalization of a classical result of Chai in the Siegel case (\cite{Chai-ordinary-isogeny-invent-math}).

\begin{theorem}[{\cite[Theorem I]{van-hoften-ordinary-hecke}}]\label{thm-Hx-dense}
Assume that the ordinary locus $S_K^{\ord}$ is nonempty. For any point $x\in S_K^{\ord}$, the set $\Hcal^p(x)$ is Zariski dense in $S_{K}$.
\end{theorem}
Since the formation of the flag space $\Flag(S_K)$ is functorial, we obtain a correspondence similar to \eqref{hecke-diag} on flag spaces. In particular, the Hecke algebra $\TT_{K^p}(\mathbf{G})$ also acts on 0-cycles of the flag space $\Flag(S_K)$. For $g\in \mathbf{G}(\AA_f^p)$, we define the operator $T_{K^p,g}$ similarly to \eqref{TKp}.

\subsection{Automorphic vector bundles} \label{sec-autom-vb}
Let $\lambda\in X^*(\mathbf{T})$ be a character. View $\lambda$ as a character of $\mathbf{B}$ and define a $\mathbf{P}$-representation
\begin{equation}
    \mathbf{V}_I(\lambda) = \ind_{\mathbf{B}}^{\mathbf{P}}(\lambda).
\end{equation}
By applying the $\mathbf{P}$-torsor $\Ical_{\mathbf{P}}$ on $\Sh_K(\mathbf{G},\mathbf{X})$, one can attach to this representation a vector bundle $\Vcal_I(\lambda)$ over $\Sh_K(\mathbf{G},\mathbf{X})$, called automorphic vector bundle. Furthermore, it extends to a vector bundle over the integral model $\Scal_K$, that we continue to denote by $\Vcal_I(\lambda)$. On the special fiber $S_K$, this vector bundle coincides with the pullback via $\zeta\colon S_K\to \GZip^\mu$ of the vector bundle $\Vcal_I(\lambda)$ on the stack $\GZip^\mu$ defined in section \ref{sec-vb-Gzip} (implicitly, we identify the character groups $X^*(\textbf{T})$ and $X^*(T)$). For any field $F$ which is an $\Ocal_{\mathbf{E}_\pfr}$-algebra, the elements of
\begin{equation}
    H^0(\Scal_K\otimes_{\Ocal_{\mathbf{E}_\pfr}} F,\Vcal_I(\lambda))
\end{equation}
are called automorphic forms of weight $\lambda$ and level $K$ with coefficients in $F$. As explained in \cite{Imai-Koskivirta-vector-bundles}, one can attach to $\lambda\in X^*(T)$ a line bundle $\Vcal_{\flag}(\lambda)$ on the stack $\GF^\mu$ such that 
\begin{equation}\label{push-Vflag}
    \pi_*(\Vcal_{\flag}(\lambda))=\Vcal_I(\lambda).
\end{equation}
Similarly, denote again by $\Vcal_{\flag}(\lambda)$ the line bundle $\zeta_{\flag}^*(\Vcal_{\flag}(\lambda))$ on $\Flag(S_K)$ obtained by pullback via the map $\zeta_{\flag}$ of diagram \eqref{diag-flag}. Since $\zeta$ is smooth, \eqref{push-Vflag} also holds for $S_K$, i.e. the push-forward of $\Vcal_{\flag}(\lambda)$ on $\Flag(S_K)$ coincides with $\Vcal_I(\lambda)$ on $S_K$. In particular, we have an identification
\begin{equation}\label{ident-flag-VI}
    H^0(S_K,\Vcal_I(\lambda)) = H^0(\Flag(S_K),\Vcal_{\flag}(\lambda)).
\end{equation}
For characters $\lambda,\lambda'\in X^*(T)$, one has $\Vcal_{\flag}(\lambda)\otimes \Vcal_{\flag}(\lambda')=\Vcal_{\flag}(\lambda+\lambda')$. Moreover, there is a natural map of $L$-representations
\begin{equation}\label{VI-add}
    V_I(\lambda)\otimes V_I(\lambda')\to V_I(\lambda+\lambda'), \quad (f,f')\mapsto ff'.
\end{equation}
It induces a map of vector bundles $\Vcal_I(\lambda)\otimes \Vcal_I(\lambda')\to \Vcal_I(\lambda+\lambda')$. 

In several previous papers (\cite{Goldring-Koskivirta-global-sections-compositio, Goldring-Koskivirta-divisibility, Koskivirta-vanishing-Siegel,Koskivirta-Hilbert-strata}), we studied the set of weights $\lambda\in X^*(\mathbf{T})$ such that there exists nonzero automorphic forms of weight $\lambda$. This problem can be investigated for various coefficient fields. For any field $F$ which is a $\Ocal_{\mathbf{E}_\pfr}$-algebra, define:
\begin{equation}\label{CKF} 
    C_K(F) \colonequals \{\lambda\in X^*(\mathbf{T}) \ | \ H^0(\Scal_{K}\otimes_{\Ocal_{\mathbf{E}_\pfr}} F, \Vcal_I(\lambda)) \neq 0 \}.
\end{equation}
Write $\Ccal(F)$ for the saturation of $C_K(F)$ (since the saturation is independent of the level $K$ by \cite[Corollary 1.5.3]{Koskivirta-automforms-GZip}, we drop it from the notation). Since $\Vcal_I(\lambda)=0$ when $\lambda$ is not $I$-dominant, $\Ccal(F)$ is always contained in $X^*_{+,I}(T)$. The case $F=\CC$ was studied in \cite{Goldring-Koskivirta-GS-cone} (we implicitly choose an isomorphism $\overline{\QQ}_p\simeq \CC$ to view $\CC$ as an $\Ocal_{\mathbf{E}_\pfr}$-algebra). It seems known to experts (although we are not aware of any reference in the literature) that $\Ccal(\CC)$ coincides with the Griffiths--Schmid cone, defined as follows:
\begin{equation}
\Ccal_{\GS}=\left\{ \lambda\in X^{*}(\mathbf{T}) \ \relmiddle| \
\parbox{6cm}{
$\langle \lambda, \alpha^\vee \rangle \geq 0 \ \textrm{ for }\alpha\in \Phi_{\mathbf{L}}^{+}, \\
\langle \lambda, \alpha^\vee \rangle \leq 0 \ \textrm{ for }\alpha\in \Phi^+ \setminus \Phi_{\mathbf{L}}^{+}$}
\right\}.
\end{equation}
We proved in \loccit the inclusion $\Ccal(\CC)\subset \Ccal_{\GS}$ for any Hodge-type Shimura varieties. The main object of study in this article is the set $\Ccal(\overline{\FF}_p)$, which has a much more mysterious structure. By the identification \eqref{ident-flag-VI}, one can interpret $\Ccal(\overline{\FF}_p)$ as the cone of effective divisors on $\Flag(S_K)$ that are associated to automorphic line bundles $\Vcal_{\flag}(\lambda)$.

\subsection{Stable base locus of automorphic bundles}
For a line bundle $\Lcal$ on a scheme $X$, recall that the base locus $B(\Lcal)$ of $\Lcal$ is the set of points $x\in X$ such that any element of $H^0(X,\Lcal)$ vanishes at $x$. It is a closed subset of $X$. The stable base locus $\BB(\Lcal)$ of $\Lcal$ is defined as the intersection of all stable base loci of the positive powers of $\Lcal$, i.e.
\begin{equation}
    \BB(\Lcal)\colonequals \bigcap_{n\geq 1}B(\Lcal^{\otimes n}).
\end{equation}
Let $\lambda\in X^*(T)$ and consider the sheaves $\Vcal_I(\lambda)$ on $S_K$ and $\Vcal_{\flag}(\lambda)$ on $\Flag(S_K)$. We denote the base locus of $\Vcal_{\flag}(\lambda)$ in $\Flag(S_K)$ by $B(\lambda)$ and the stable base locus by $\BB(\lambda)$. Say that a subset $Z\subset S_K$ or $Z\subset \Flag(S_K)$ is stable by prime-to-$p$ Hecke operators if for any point $x\in Z$ and any $g\in \mathbf{G}(\AA_f^p)$, the support of the $0$-cycle $T_{K^p,g}(x)$ is entirely contained in $Z$. We proved in \cite[Theorem 3.6.2]{Koskivirta-vanishing-Siegel} the following result:

\begin{theorem}\label{thm-B-stable}
    The stable base locus $\BB(\lambda)$ is stable by prime-to-$p$ Hecke operators.
\end{theorem}

Next, we define the base locus $B_I(\lambda)$ of the vector bundle $\Vcal_I(\lambda)$ as the set of points $x\in S_K$ such that for any $f\in H^0(S_K,\Vcal_I(\lambda))$, the section $f$ vanishes at $x$. Note that $\Vcal_I(\lambda)$ is a vector bundle, so this condition means that all the coordinate functions of $f$ (in some local basis) vanish at $x$. The base locus $B_I(\lambda)$ is a closed subscheme of $S_K$. Similarly, define the stable base locus $\BB_I(\lambda)$ of $\Vcal_I(\lambda)$ as the intersection
\begin{equation}
    \BB_I(\lambda)\colonequals \bigcap_{n\geq 1} B_I(n\lambda). 
\end{equation}
Since $\pi_{*}(\Vcal_{\flag}(\lambda))=\Vcal_I(\lambda)$, we can rewrite the base locus $B_I(\lambda)$ and the stable base locus $\BB_I(\lambda)$ in terms of the corresponding sets for the line bundle $\Vcal_{\flag}(\lambda)$. Denote by $\pi_K\colon \Flag(S_K)\to S_K$ the natural projection. Then, we have the following:
\begin{align*}
    B_I(\lambda)&=\{x\in S_K \ | \ \pi_K^{-1}(x)\subset B(\lambda) \}, \\
    \BB_I(\lambda)&=\{x\in S_K \ | \ \pi_K^{-1}(x)\subset \BB(\lambda) \}.
\end{align*}
Since $\BB(\lambda)$ is stable by prime-to-$p$ Hecke operators, we deduce immediately that $\BB_I(\lambda)$ is also stable by prime-to-$p$ Hecke operators. Using this observation, we can prove the following:
\begin{proposition}\label{prop-BI}
Let $S_K$ be a Hodge-type Shimura variety with nonempty ordinary locus. For any $\lambda\in X^*(T)$, the following are equivalent:
\begin{equivlist}
    \item $\BB_I(\lambda)$ contains an ordinary point.
    \item For any $n\geq 1$, we have $H^0(S_K,\Vcal_I(n\lambda))=0$.
    \item One has $\lambda \notin \Ccal(\overline{\FF}_p)$.
\end{equivlist}
\end{proposition}

\begin{proof}
The equivalence of (ii) and (iii) follows from the definition of $\Ccal(\overline{\FF}_p)$, and the implication (ii) $\Rightarrow$ (i) is obvious. Conversely, assume that $\BB_I(\lambda)$ intersects the ordinary locus. By stability under prime-to-$p$ Hecke operators, and using the density result of prime-to-$p$ Hecke orbits of van Hoften (Theorem \ref{thm-Hx-dense}), we obtain that $\BB_I(\lambda)=S_K$. This shows the result.
\end{proof}

In other words, the above result shows that if some power of $\Vcal_{\flag}(\lambda)$ is effective (i.e. admits a nonzero section), then $\BB_I(\lambda)$ is entirely contained in the non-ordinary locus. In general, we believe that when $\lambda\in \Ccal(\overline{\FF}_p)$, the locus $\BB_I(\lambda)$ is even empty, but the weaker statement of 
 Proposition \ref{prop-BI} is enough for our purpose. This proposition has the following corollary. Let $X_H$, $X_G$ denote the mod $p$ special fibers of a Hilbert--Blumenthal and Siegel Shimura varieties respectively, with good reduction at $p$, as in the setting of section \ref{hb-var-sec}. Fix a symplectic embedding $u\colon \mathbf{H}\to \mathbf{G}$ as in \eqref{u-embedQ}, which induces an embedding $\widetilde{u}\colon X_H\to X_G$ (for an appropriate choice of level structures).

\begin{corollary}\label{cor-pb-nonzero}
Assume that $H^0(X_G,\Vcal_I(\lambda))\neq 0$. Then there exists $m\geq 1$ and a nonzero Siegel automorphic form $f\in H^0(X_G,\Vcal_I(m\lambda))$ such that the pullback
\[\widetilde{u}^*(f) \in H^0(X_H, \widetilde{u}^*(\Vcal_I(m\lambda)))\]
is nonzero.
\end{corollary}

\begin{proof}
Since $H^0(X_G,\Vcal_I(\lambda))\neq 0$, the stable base locus $\BB_I(\lambda)$ does not intersect the ordinary locus. Since the ordinary locus of Hilbert--Blumenthal Shimura varieties is nonempty, we can choose an ordinary point $x\in \widetilde{u}(X_H)$. As $x\notin \BB_I(\lambda)$, we can find $m\geq 1$ and a section $f\in H^0(X_G,\Vcal_I(m\lambda))$ which does not vanish at $x$. The result follows.
\end{proof}

\begin{rmk}
For the proof of the Cone Conjecture (Theorem \ref{main-thm}) carried out in this paper, it would be highly desirable to do away with the multiple $m\lambda$ of $\lambda$ that appears in Corollary \ref{cor-pb-nonzero}, even at the expense of changing the totally real field $\mathbf{F}$. We believe that the answer to the following question is always "yes":
\begin{question}\label{quest-Hilbert}
Assume that $H^0(X_G,\Vcal_I(\lambda))\neq 0$. Does there exists
\begin{bulletlist}
    \item a totally real field $\mathbf{F}$ of degree $n$ such that $p$ is split in $\mathbf{F}$,
    \item a Hilbert--Blumenthal Shimura variety $X_H$ with good reduction at $p$, attached to $\mathbf{F}$, with a symplectic embedding $\widetilde{u}\colon X_H \to X_G$ of Shimura varieties,
    \item a nonzero automorphic form $f\in H^0(X_G,\Vcal_I(\lambda))$,
\end{bulletlist}
such that the pullback $\widetilde{u}^*(f) \in H^0(X_H, \widetilde{u}^*(\Vcal_I(\lambda)))$ is nonzero ?
\end{question}
If this question has a positive answer, then we will see that it is possible to give more precise version of the Cone Conjecture, where we have control over the weight (see Remark \ref{rmk-precise-weight}). To show that a triple $(\mathbf{F},\widetilde{u},f)$ as in Question \ref{quest-Hilbert} always exists, it would be enough to prove that the union of all Hilbert--Blumenthal Shimura varieties embedded into $X_G$ (with the condition that $p$ splits in $\mathbf{F}$) is Zariski dense. Another strategy would be to use Hecke operators to move $f$ around so that its Hecke transform does not vanish at a specified point.
\end{rmk}

\subsection{Torsors}
Let $S_K$ be the special fiber of a Hodge-type Shimura variety with good reduction such that the ordinary locus $S_K^{\ord}$ is nonempty. Recall that $S_K$ is endowed with a $P$-torsor $I_P$. In particular, $I_L\colonequals I_P/R_{\mathrm{u}}(P)$ is an $L$-torsor on $S_K$. Write
\begin{equation}
    \beta\colon Y_K\to S_K
\end{equation}
for the geometric object corresponding to this $L$-torsor. Here, $Y_K$ is a $k$-scheme endowed with an action of $L$ and $\beta$ is a smooth morphism. Let $Y^{\ord}_K$ denote the preimage of the ordinary locus $S_K^{\ord}$ by $\beta$. The map $\zeta\colon S_K\to \GZip^\mu$ restricts to a map 
\begin{equation}
    \zeta^{\ord}\colon S_K^{\ord}\to \Ucal^{\ord}.
\end{equation}
Since $\Ucal^{\ord}\simeq B(L(\FF_p))$, we obtain a $L(\FF_p)$-principal bundle over $S_K^{\ord}$, that gives rise to $Y^{\ord}_K$ when we extend it to an $L$-torsor. Write
\begin{equation}
    \beta_{L(\FF_p)}\colon Y^{\ord}_{L(\FF_p),K}\to S^{\ord}_K
\end{equation}
for the geometric object corresponding to this $L(\FF_p)$-principal bundle (as explained in section \ref{sec-torsors}). By the above discussion, we have an isomorphism
\begin{equation}\label{isom-torsor-Y}
    \left[ L(\FF_p) \backslash \left(Y^{\ord}_{L(\FF_p),K} \times L\right) \right] \simeq Y^{\ord}_K
\end{equation}
where $L(\FF_p)$ acts diagonally on $Y^{\ord}_{L(\FF_p),K} \times L$.

\section{\texorpdfstring{Automorphic forms in characteristic $p$}{}}\label{sec-autom}

\subsection{Trivial part of an automorphic vector bundle} \label{sec-triv-vb}
Let $S_K$ be the special fiber of a Hodge-type Shimura variety at a place of good reduction, as in section \ref{sec-HT-SV}. We fix a symplectic embedding into a Siegel-type Shimura variety, whose special fiber we denote by $\widetilde{u}\colon S_K\to \overline{\Acal}_{n,K_0}$ (for an appropriate choice of level structures). For technical reasons, we need to make the following assumptions:
\begin{assumption}\label{assume-SK-irred} \ 
\begin{assertionlist}
    \item The classical ordinary locus $S_K^{\ord}$ is nonempty.
    \item For any connected component $S\subset S_K$, the non-ordinary locus $S\setminus S^{\ord}$ is irreducible.
    \item The classical Hasse invariant $\Ha$ pulled back to $S$ via $\widetilde{u}$ has multiplicity one along the non-ordinary locus.
\end{assertionlist}
\end{assumption}

All three assumptions are satisfied in the case of the Siegel-type Shimura variety $\overline{\Acal}_{n,K}$, by Proposition \ref{prop-non-ord}. Let $\zeta\colon S_K\to \GZip^\mu$ be the map \eqref{zeta-Shimura}. Recall that $S_K^{\ord}=\zeta^{-1}(\Ucal^{\ord})$, where $\Ucal^{\ord}$ is the unique open stratum of $\GZip^\mu$. Since $\Ucal^{\ord}\simeq B(L(\FF_p))$ by \eqref{Ucal-ord}, the $k$-scheme $S_K^{\ord}$ is endowed naturally with an $L(\FF_p)$-torsor corresponding to the map $\zeta^{\ord}\colon S_K^{\ord}\to \Ucal^{\ord}$.

Let $\lambda\in X^*(T)$ be a character and consider the associated automorphic vector bundle $\Vcal_I(\lambda)$, attached to the $L$-representation $V_I(\lambda)$. View $V_I(\lambda)$ as an $L(\FF_p)$-representation and write $V_I(\lambda)^{L(\FF_p)}$ for the subspace of $L(\FF_p)$-invariant vectors. Let $\Vcal_I(\lambda)_{\triv}$ be the corresponding vector bundle on $S_K^{\ord}$, as explained in section \ref{sec-fin-et}. 

\begin{definition}
Let $f\in H^0(S_K,\Vcal_I(\lambda))$ be an automorphic form of weight $\lambda$ and level $K$. We say that $f$ is of trivial-type if the restriction $f^{\ord}\colonequals f|_{S_K^{\ord}}$ lies in the subspace $H^0(S_K^{\ord},\Vcal_I(\lambda)_{\triv})\subset H^0(S_K^{\ord},\Vcal_I(\lambda))$. 
\end{definition}

We denote by $H^0(S_K,\Vcal_I(\lambda))_{\triv}$ the subspace of trivial-type automorphic forms of weight $\lambda$ and level $K$. In other words, one has the identity
\begin{equation}
H^0(S_K,\Vcal_I(\lambda))_{\triv} = H^0(S_K,\Vcal_I(\lambda)) \cap H^0(S^{\ord}_K,\Vcal_I(\lambda)_{\triv})
\end{equation}
where this intersection is taken inside the ambiant space $H^0(S^{\ord}_K,\Vcal_I(\lambda))$. If $f$ is an automorphic form that arises by pullback from the stack of $G$-zips via the map $\zeta\colon S_K\to \GZip^{\mu}$, then it is automatically of trivial-type since sections of $\Vcal_I(\lambda)$ and $\Vcal_I(\lambda)_{\triv}$ over $B(L(\FF_p))$ coincide (see \eqref{BH-triv}). Therefore, we have inclusions
\begin{equation}\label{inclu-triv}
    H^0(\GZip^{\mu},\Vcal_I(\lambda))\subset H^0(S_K,\Vcal_I(\lambda))_{\triv} \subset H^0(S_K,\Vcal_I(\lambda)).
\end{equation}

At this point, one might wonder if there exist any automorphic forms of trivial-type other than the ones arising by pullback from the stack $\GZip^\mu$. Surprisingly, the answer is that there many such automorphic forms. More precisely, for any $f\in H^0(S^{\ord}_K,\Vcal_I(\lambda))$, we will define in section \ref{sec-sym-transf} below the norm of $f$ as an automorphic form $\Norm_{L(\FF_p)}(f)\in H^0(S_K^{\ord},\Vcal_I(N\lambda)_{\triv})$ where $N=|L(\FF_p)|$. One shows easily that this norm is always nonzero when $f\neq 0$. Then, one can multiply (if necessary) $\Norm_{L(\FF_p)}(f)$ by a sufficiently large power $m\geq 1$ of the classical Hasse invariant so that $\Ha^m \Norm_{L(\FF_p)}(f)$ extends to $S_K$. This shows that any automorphic form $f$ of weight $\lambda$ gives rise to an automorphic form of trivial-type of weight $N\lambda+m (p-1)\lambda_\omega$ for some $m\geq 0$. We will need the following, stronger result.

\begin{theorem}\label{thm-triv-lambda}
Assume that $H^0(S_K,\Vcal_I(\lambda))\neq 0$ and write $N\colonequals |L(\FF_p)|$. There exists an integer $1\leq d\leq N$ such that $H^0(S_K,\Vcal_I(d\lambda))_{\triv}\neq 0$.
\end{theorem}

The proof of Theorem \ref{thm-triv-lambda} will occupy most of section \ref{sec-autom}. Our ultimate goal is to show that if $H^0(S_K,\Vcal_I(\lambda))\neq 0$, then there exists $d\geq 1$ such that $H^0(\GZip^\mu,\Vcal_I(d\lambda))\neq 0$ (Theorem \ref{main-thm}). Therefore, in view of the inclusions \eqref{inclu-triv}, the above result is a first step towards the proof of the Cone Conjecture.

\begin{rmk}
We end this paragraph by giving an example of an automorphic form that is not of trivial-type. Consider the the case of the modular curve, where $G=\GL_{2,\FF_p}$ and $L=T$. Write $V_k$ for the one-dimensional representation given by the character $\lambda^k_\omega$. Then, one sees easily that $V_k^{T(\FF_p)}=0$ unless $k$ is a multiple of $p-1$. In particular, any nonzero modular form of weight $1$ is not of trivial-type for $p\geq 3$.
\end{rmk}

\subsection{Automorphic forms and torsors}

Recall that we have an $L$-torsor $\beta\colon Y_K\to S_K$ and its restriction to the ordinary locus $\beta^{\ord}\colon Y^{\ord}_K\to S^{\ord}_K$. Furthermore, $\beta^{\ord}$ descends to an $L(\FF_p)$-principal bundle $\beta_{L(\FF_p)}\colon Y^{\ord}_{L(\FF_p),K}\to S_K^{\ord}$.

\subsubsection{Sections via $L$-torsors}\label{sec-sections-L-torsors}
By section \ref{sec-torsors}, an element $f\in H^0(S_K,\Vcal_I(\lambda))$ (for $\lambda\in X^*(T)$) is the same as a regular $L$-equivariant map $f\colon Y\to V_I(\lambda)$. Since $V_I(\lambda)$ is the space of regular maps $\gamma\colon P\to \AA^1$ satisfying $\gamma(xb)=\lambda(b)^{-1}\gamma(x)$ for all $x\in P$, $b\in B$, the map 
\begin{equation} \label{ftilde}
    \widetilde{f}\colon Y_K\times P\to \AA^1, \quad (y,x)\mapsto [f(y)](x)
\end{equation}
is regular and satisfies the relations
\begin{align*}
    \widetilde{f}(y,xb) &= \lambda^{-1}(b)\widetilde{f}(y,x), \quad b\in B, \\
    \widetilde{f}(a\cdot y,ax) &= \widetilde{f}(y,x), \quad a\in L.  
\end{align*}
This yields a bijective correspondence between elements of $H^0(S_K,\Vcal_I(\lambda))$ and regular maps $\widetilde{f}\colon Y_K\times P \to \AA^1$ satisfying the above relations. Similarly, sections of $\Vcal_I(\lambda)$ over $S_K^{\ord}$ correspond to regular maps $\widetilde{f}\colon Y^{\ord}_K\times P \to \AA^1$ satisfying the same relations.

\subsubsection{Sections via $L(\FF_p)$-torsors}
For any $L(\FF_p)$-representation $\rho\colon L(\FF_p)\to \GL(V)$, a global section $f\in H^0(S_K^{\ord},\Vcal(\rho))$ corresponds to an $L(\FF_p)$-equivariant regular map on $Y^{\ord}_{L(\FF_p),K}$ with values in $V$. Let $\lambda\in X^*(T)$ and view the $L$-representation $V_I(\lambda)$ as an $L(\FF_p)$-representation. Then, we may also view a section $f\in H^0(S_K^{\ord},\Vcal_I(\lambda))$ as a regular map
\begin{equation}
    \widetilde{f}\colon Y^{\ord}_{L(\FF_p),K}\times P \to \AA^1
\end{equation}
satisfying for all $y\in Y^{\ord}_{L(\FF_p),K}$ and all $x\in P$ the relations
\begin{align*}
    \widetilde{f}(y,xb) &= \lambda^{-1}(b)\widetilde{f}(y,x) \quad b\in B \\
    \widetilde{f}(a\cdot y,ax) &= \widetilde{f}(y,x), \quad a\in L(\FF_p).
\end{align*}
This point of view is related to the one in section \ref{sec-sections-L-torsors} above in the following way. If $\widetilde{f}\colon Y^{\ord}_{L(\FF_p),K}\times P \to \AA^1$ is a map as above, we can define a map
\begin{equation}
   \widetilde{f}_0 \colon \left[L(\FF_p) \backslash \left(Y^{\ord}_{L(\FF_p),K} \times L\right)\right] \times P \to \AA^1
\end{equation}
by setting $\widetilde{f}_0((y,a),x)\colonequals \widetilde{f}(y,ax)$. One sees immediately that the invariance properties of $\widetilde{f}$ imply that $\widetilde{f}_0$ is a well-defined map. If we identify $Y_K^{\ord}$ and $\left[  L(\FF_p) \backslash \left(Y^{\ord}_{L(\FF_p),K} \times L\right) \right]$ via \eqref{isom-torsor-Y}, then $\widetilde{f}_0$ yields a map $\widetilde{f}_0\colon Y^{\ord}_K\times P \to \AA^1$
as in section \ref{sec-sections-L-torsors}.

\subsection{Symmetric transforms of ordinary automorphic forms} \label{sec-sym-transf}
Let $1\leq d \leq |L(\FF_p)|$ be an integer. By the discussion in section \ref{sec-norm-map}, we can define for any section $f\in H^0(S_K^{\ord},\Vcal_I(\lambda))$ its $d$-th symmetric transform $\sfr^{(d)}_{L(\FF_p)}(f)$, which lies in the space
\begin{equation}
    H^0(S_K^{\ord}, \Sym^d(\Vcal_I(\lambda))_{\triv}).
\end{equation}
The natural map of $L$-representations $\Sym^d(V_I(\lambda))\to V_I(d\lambda)$ afforded by \eqref{VI-add} gives rise to an $L(\FF_p)$-linear map
\begin{equation}
    \Sym^d(V_I(\lambda))^{L(\FF_p)}\to V_I(d\lambda)^{L(\FF_p)},
\end{equation}
and hence a map $ \Sym^d(\Vcal_I(\lambda))_{\triv}\to \Vcal_I(d\lambda)_{\triv}$. Therefore, we may view $\sfr^{(d)}_{L(\FF_p)}(f)$ as an element of the space $H^0(S^{\ord}_K,\Vcal_I(d\lambda)_{\triv})$. This construction gives rise to a family of non-linear mappings
\begin{equation}\label{ord-Sd}
    \sfr^{(d)}_{L(\FF_p)}\colon H^0(S_K^{\ord},\Vcal_I(\lambda))\to H^0(S_K^{\ord},\Vcal_I(d\lambda)_{\triv}), \quad 1\leq d\leq |L(\FF_p)|.
\end{equation}

Note that the vector bundle $\Vcal_I(d\lambda)_{\triv}$ only makes sense on the ordinary locus, it does not extend to $S_K$ since it is not attached to an $L$-representation. However, we may compose the map \eqref{ord-Sd} with the natural injection 
\begin{equation}
    H^0(S_K^{\ord},\Vcal_I(d\lambda)_{\triv})\subset H^0(S_K^{\ord},\Vcal_I(d\lambda)).
\end{equation}
We thus obtain a map attaching to each automorphic form (over $S_K^{\ord}$) of weight $\lambda$ an automorphic form (over $S_K^{\ord}$) of weight $d\lambda$. We may ask if this map induces by restriction a map
\begin{equation}\label{SK-norm}
    \sfr^{(d)}_{L(\FF_p)}\colon H^0(S_K,\Vcal_I(\lambda))\to H^0(S_K,\Vcal_I(d\lambda)).
\end{equation}
We believe that this is true, but we were not able to show it.

\subsection{Rings of automorphic forms} \label{sec-rings-R}

We fix a connected component $S\subset S_K$ and consider sections in $H^0(S,\Vcal_I(\lambda))$. Since any nonzero automorphic form on $S_K$ is nonzero on some connected component, it will be enough for us to study the space $H^0(S,\Vcal_I(\lambda))$. We write $S^{\ord}$ for the ordinary locus of $S$. Denote by
\begin{equation}
    \beta\colon Y\to S
\end{equation}
the $L$-torsor on $S$, obtained by restricting the $L$-torsor $\beta\colon Y_K\to S_K$ to $S$. Since $\beta$ is smooth and $S,L$ are irreducible, we deduce that $Y$ is smooth and irreducible. Let 
\begin{equation}
\pi\colon \Flag(S)\to S    
\end{equation}
denote the flag space of $S$, i.e. the preimage of $S$ by the map $\pi_K\colon \Flag(S_K)\to S_K$. It is clear that $\Flag(S)$ is also connected and smooth (Lemma \ref{lem-irred-conn}). Recall that $\Flag(S)^{\ord}$ denotes the preimage of $S^{\ord}$ by $\pi$. We define the following rings:
\begin{equation}
        R^{\ord} \colonequals \bigoplus_{\lambda\in X^*(T)} H^0(S^{\ord},\Vcal_I(\lambda)) \quad \ , \quad \ R \colonequals \bigoplus_{\lambda\in X^*(T)} H^0(S,\Vcal_I(\lambda)).
\end{equation}
Similarly, there are versions of $R^{\ord}$ and $R$ for the trivial part of automorphic vector bundles. We define $H^0(S,\Vcal_I(\lambda))_{\triv}$ as the subspace of elements $f\in H^0(S,\Vcal_I(\lambda))$ such that $f|_{S^{\ord}}$ lies in $H^0(S^{\ord},\Vcal_I(\lambda)_{\triv})$. We put:
\begin{equation}
        R^{\ord}_{\triv} \colonequals \bigoplus_{\lambda\in X^*(T)} H^0(S^{\ord},\Vcal_I(\lambda)_{\triv}) \quad \ , \quad \ R_{\triv} \colonequals \bigoplus_{\lambda\in X^*(T)} H^0(S,\Vcal_I(\lambda))_{\triv}.
\end{equation}
The natural map $\Vcal_I(\lambda)_{\triv}\otimes \Vcal_{I}(\lambda')_{\triv}\to \Vcal_I(\lambda+\lambda')_{\triv}$ (for $\lambda, \lambda'\in X^*(T)$) endows the rings $R^{\ord}_{\triv}$ and $R_{\triv}$ with a natural structure as $k$-algebras. One has obvious inclusion maps of $k$-algebras:
\begin{equation}
    \xymatrix@M=7pt{
R_{\triv} \ar@{^{(}->}[r] \ar@{^{(}->}[d] & R^{\ord}_{\triv} \ar@{^{(}->}[d]  \\
R \ar@{^{(}->}[r] & R^{\ord} 
    }
\end{equation}
We may also view $R$, $R^{\ord}$ as the "automorphic Cox rings" of $\Flag(S)$ and $\Flag(S)^{\ord}$ respectively. By this, we mean that the relation $\pi_*(\Vcal_{\flag}(\lambda))=\Vcal_I(\lambda)$ implies that $R$, $R^{\ord}$ identify respectively with the Cox rings
\begin{align*}
    \Cox(\Flag(S))&\colonequals \bigoplus_{\lambda\in X^*(T)} H^0(\Flag(S),\Vcal_{\flag}(\lambda))  \quad \textrm{and}\\ \Cox(\Flag(S)^{\ord})&\colonequals \bigoplus_{\lambda\in X^*(T)} H^0(\Flag(S)^{\ord},\Vcal_{\flag}(\lambda)).
\end{align*}
See \cite{Koskivirta-Mori-dream-space} for more details on the Cox rings of stacks of $G$-zips and $G$-zip flags. Note that $\Flag(S)^{\ord}$ is cut out inside $\Flag(S)$ by the section $\pi^*(\Ha)$, the pullback of the classical Hasse invariant via the map $\pi\colon \Flag(S)\to S$. The complement of $\Flag(S)^{\ord}$ inside $\Flag(S)$ is a codimension one, irreducible closed subvariety (Proposition \ref{flag-non-ord-irred}). The advantage of working with a connected component $S\subset S_K$ is the following lemma:

\begin{lemma}\label{lemma-integral-dom-R}
The ring $R$ is an integral domain.
\end{lemma}

\begin{proof}
When $f,f'$ are nonzero homogeneous elements, then it is clear that $ff'\neq 0$. Indeed, let $f,f'$ be elements of $H^0(S,\Vcal_I(\lambda))$ and $H^0(S,\Vcal_I(\lambda'))$ respectively. We may view $f,f'$ as sections of $\Vcal_{\flag}(\lambda)$ and $\Vcal_{\flag}(\lambda')$ on the flag space $\Flag(S)^{\ord}$ respectively. Since $\Flag(S)$ is smooth and connected, $ff'$ must be nonzero on an open subset of $\Flag(S)^{\ord}$. Next, we examine the general case. For a nonhomogeneous element $f\in R$, we write $f=\sum_{\lambda} f_\lambda$ (for $\lambda\in X^*(T)$) and consider the function $\widetilde{f}_\lambda \colon Y\times P\to \AA^1$ that corresponds to $f_\lambda$, as explained in \eqref{ftilde}. The map
\begin{equation}
    R\to k[Y\times P], \quad \sum_{\lambda} f_\lambda \mapsto \sum_{\lambda} \widetilde{f}_\lambda
\end{equation}
is an injective ring homomorphism. Since $Y$ is an integral scheme, we deduce that $R$ is an integral domain.

\end{proof}

\begin{rmk}
It follows from the proof of Lemma \ref{lemma-integral-dom-R} that $R$ identifies with a subring of $k[Y\times P]$. One can show that it coincides with the subring of functions $f\colon Y\times P\to \AA^1$ satisfying for all $(y,x)\in Y\times P$ the following relations:
\begin{align*}
    \widetilde{f}(y,xu) &= \widetilde{f}(y,x), \qquad u\in R_{\mathrm{u}}(B) \\
    \widetilde{f}(a\cdot y,ax) &= \widetilde{f}(y,x), \qquad a\in L.  
\end{align*}
\end{rmk}

We denote by $F$ the field of fractions of $R$. Write simply $\Ha$ instead of $\pi^*(\Ha)$ for the classical Hasse invariant on $S$ pulled back to the flag space $\Flag(S)$. One clearly has $R^{\ord}=R\left[\frac{1}{\Ha}\right]$, so we deduce that $R^{\ord}$ is also an integral domain with field of fractions $F$. In particular, $R^{\ord}_{\triv}$ and $R_{\triv}$ are also integral domains. We write $F_{\triv}$ for the field of fractions of $R_{\triv}$. We continue our study of the various rings of automorphic forms and prove the following elementary properties:

\begin{proposition} \ 
    \begin{assertionlist}
        \item One has $\Frac(R^{\ord}_{\triv})=\Frac(R_{\triv})=F_{\triv}$.
        \item The ring extension $R^{\ord}_{\triv}\subset R^{\ord}$ is integral.
        \item The field extension $F/F_{\triv}$ is algebraic.
    \end{assertionlist}
\end{proposition}

\begin{proof}
Since $\Ha$ comes by pullback from the stack of $G$-zips, it is an automorphic form of trivial-type. Hence, we also have the identity $R^{\ord}_{\triv}=R_{\triv}\left[\frac{1}{\Ha}\right]$, which shows assertion (1). Next, if $f\in H^0(S^{\ord},\Vcal_I(\lambda))$ is a section over $S^{\ord}$, then the symmetric transform $\sfr^{(d)}_{L(\FF_p)}(f)$ lies in $H^0(S^{\ord},\Vcal_I(d\lambda)_{\triv})$ for any $1\leq d \leq N$, where $N=|L(\FF_p)|$. Since we have the relation
\begin{equation}
    f^n + \sfr^{(1)}_{L(\FF_p)}(f) f^{n-1} + \sfr^{(2)}_{L(\FF_p)}(f) f^{n-2} + \dots + \sfr^{(N)}_{L(\FF_p)}(f) = 0,
\end{equation}
we deduce that the extension $R^{\ord}_{\triv}\subset R^{\ord}$ is integral. Finally, assertion (3) follows immediately from (1) and (2).
\end{proof}

\subsection{The ordinary valuation}

Recall that the non-ordinary locus $S\setminus S^{\ord}$ is irreducible, cut out by the restriction to $S$ of the classical Hasse invariant $\Ha\in H^0(S,\omega^{p-1})$. Similarly, consider the non-ordinary locus $Y\setminus Y^{\ord}$ of the $L$-torsor $\beta\colon Y\to S$. Since $\beta$ is smooth with irreducible fibers, we deduce that $Y\setminus Y^{\ord}$ is irreducible, cut out (with multiplicity one) by the pull-back of $\Ha$ to $Y$. This gives a valuation $\val_{\ord}$ on the field of functions $k(Y)$. Similarly, if we set
\begin{equation}
    F_0 \colonequals k(Y\times P),
\end{equation}
then the irreducible divisor $(Y\setminus Y^{\ord})\times P$ gives rise to a discrete valuation
\begin{equation}
    \val_{\ord}\colon F_0\to \ZZ\cup \{\infty\}.
\end{equation}
Again, the Hasse invariant $\Ha$ (pulled back to $Y\times P$) is a uniformizer for this valuation. We consider the field extensions
\begin{equation}
    F_{\triv}\subset F\subset F_0
\end{equation}
and denote again by $\val_{\ord}$ the discrete valuation obtained by restricting $\val_{\ord}$ to each of these fields. Write
\begin{equation}
    \widetilde{R}_{\triv}\subset \widetilde{R}\subset \widetilde{R}_0
\end{equation}
for the valuation rings with respect to $\val_{\ord}$ in the fields $F_{\triv}$, $F$, $F_0$ respectively. 
Clearly $R\subset \widetilde{R}$ and $R_{\triv}\subset \widetilde{R}_{\triv}$. Since the Hasse invariant $\Ha$ is a trivial-type automorphic form, it is a uniformizer of all three rings $\widetilde{R}_{\triv}$, $\widetilde{R}$, $\widetilde{R}_0$. In particular, the ring extension $\widetilde{R}_{\triv}\subset \widetilde{R}$ is inert, in the sense that $\Ha$ remains prime in the larger ring. Write
\begin{equation}
   (\widetilde{R}_{\triv})^{\rm cl} \subset F
\end{equation}
for the integral closure of $\widetilde{R}_{\triv}$ inside $F$.

\begin{proposition}
One has $(\widetilde{R}_{\triv})^{\rm cl}=\widetilde{R}$.
\end{proposition}

\begin{proof}
Since $\widetilde{R}$ is a discrete valuation ring, it is integrally closed, so we find immediately that $(\widetilde{R}_{\triv})^{\rm cl} \subset \widetilde{R}$. We consider a subextension $F'\subset F$ which is finite over $F_{\triv}$. The restriction of $\val_{\ord}$ to $F'$ gives a discrete valuation on $F'$; denote by $\widetilde{R}'$ its valuation ring. The integral closure of $\widetilde{R}$ in $F'$ is $\Ocal'\colonequals F'\cap (\widetilde{R}_{\triv})^{\rm cl}$. It suffices to show that $\Ocal'$ coincides with $\widetilde{R}'$. Note that $\Frac(\Ocal')=\Frac(\widetilde{R}')=F'$ and $\Ocal'\subset \widetilde{R}'$. Recall that if $A,B$ are two discrete valuation rings with the same field of fractions such that $A\subset B$, then $A$ and $B$ coincide. Therefore, it suffices to show that $\Ocal'$ is a discrete valuation ring. Since it is the integral closure of a Dedekind domain in a finite extension, it is a Dedekind domain. Furthermore, it has only finitely many maximal ideals, hence it is a principal ideal domain. Moreover, $\Ha$ is contained in each maximal ideal of $\Ocal'$. If we denote by $\varpi_1, \dots, \varpi_r$ a set of uniformizers of the maximal ideals of $\Ocal'$, then $\Ha$ decomposes as $\Ha=u\varpi^{e_1}\dots \varpi_r^{e_r}$ for some integers $e_1,\dots, e_r\geq 1$ and a unit $u\in \Ocal'^{\times}$. If we view $R$ as a subring of $k[Y\times P]$, then we may interpret any element of $F$ as a rational function on $Y\times P$. By construction, each element $\varpi_i$ must have a positive multiplicity along the irreducible divisor $(Y\setminus Y^{\ord})\times P$. Since the classical Hasse invariant has multiplicity one along the non-ordinary locus of $S$ by Assumption \ref{assume-SK-irred}, the same holds for its pullback to $Y\times P$ because the maps $Y\times P\to Y\to S$ are smooth. This clearly implies that $\Ocal'$ admits a unique maximal ideal, and thus is a discrete valuation ring. This shows the result.
\end{proof}

\subsection{Minimal polynomial}
Next, we consider an element $f\in H^0(S,\Vcal_I(\lambda))\subset R$ and investigate the minimal polynomial of $f$ with respect to the algebraic extension $F_{\triv}\subset F$. We write $M_f(t)\in F_{\triv}[t]$ for this minimal polynomial. Note that $f$ is a root of the polynomial
\begin{equation}
   P_f(t)\colonequals t^N + \sfr^{(1)}_{L(\FF_p)}(f) t^{N-1} + \sfr^{(2)}_{L(\FF_p)}(f) t^{N-2} + \dots + \sfr^{(N)}_{L(\FF_p)}(f)
\end{equation}
whose coefficients are elements of $H^0(S^{\ord}, \Vcal_I(\lambda)_{\triv})\subset F_{\triv}$. Therefore, the minimal polynomial $M_f(t)$ of $f$ over $F_{\triv}$ is a divisor of $P_f(t)$ in $F_{\triv}[t]$. Since the extension $\widetilde{R}_{\triv}\subset \widetilde{R}$ is integral, we deduce the following:

\begin{lemma}
One has $M_f(t)\in \widetilde{R}_{\triv}[t]$.
\end{lemma}

\begin{proof}
Let $\overline{F}$ be an algebraic closure of $F$. Since $f$ is integral over $\widetilde{R}_{\triv}$, all the $\Gal(\overline{F}/F_{\triv})$-conjugates of $f$ are also integral over $\widetilde{R}_{\triv}$. It follows that the coefficients of $M_f(t)$ are integral over $\widetilde{R}_{\triv}$. Since $\widetilde{R}_{\triv}$ is a discrete valuation ring, it is integrally closed, so the result follows.
\end{proof}

Recall that we have field extensions $F_{\triv}\subset F\subset F_0$. However, since the construction of the symmetric transforms $\sfr^{(d)}_{L(\FF_p)}(f)$ are inherently related to the existence of the $L(\FF_p)$-torsor $Y^{\ord}_{L(\FF_p)}$, we cannot see the roots of the polynomial $P_f(t)$ (other than $f$) in these fields. In order to decompose $P_f(t)$, one could consider the natural projection
\begin{equation}
  Y^{\ord}_{L(\FF_p)}\times L \to L(\FF_p)\backslash \left( Y^{\ord}_{L(\FF_p)}\times L \right) \simeq Y^{\ord}
\end{equation}
Then, by construction the polynomial $P_f(t)$ decomposes into linear factors in the ring $k[Y^{\ord}_{L(\FF_p)}\times L][t]$. However, it seems to us that $Y^{\ord}_{L(\FF_p)}$ may be non-connected in general, so the ring $k[Y^{\ord}_{L(\FF_p)}\times L]$ is not an integral domain. For this reason, it is difficult to determine the roots of $P_f(t)$ as elements of a certain field extension of $F$. Instead, we argue in the following indirect way. We can decompose
\begin{equation}
    P_f(t) = M_f(t) Q(t)
\end{equation}
where $P_f(t)\in R^{\ord}_{\triv}[t]$, $M_f[t]\in \widetilde{R}_{\triv}[t]$ and $Q(t)\in F_{\triv}[t]$. Furthermore, we can embed $F$ into $k(Y\times P)$ and view all coefficients of these polynomials as rational functions on $Y\times P$. Let $D\subset Y\times P$ be any irreducible divisor, different from the non-ordinary divisor $(Y\setminus Y^{\ord})\times P$. Let $\val_D$ denote the discrete valuation attached to $D$ on $k(Y\times P)$ and let $\widetilde{R}_D\subset k(Y\times P)$ denote the valuation ring of $\val_D$. Since the coefficients of $P_f(t)$ are all regular maps on $Y^{\ord}\times P$, we have $P_f(t)\in \widetilde{R}_D[t]$. Recall the following elementary lemma:

\begin{lemma}
Let $A$ be a unique factorization domain with fraction field $K$. If $P,Q,R\in K[t]$ are monic polynomials such that $P\in A[t]$ and $P=QR$, then $Q,R$ lie in $A[t]$.
\end{lemma}

It follows from this that $M_f(t)\in \widetilde{R}_D[t]$. Since $D$ was arbitrary, we deduce that the coefficients of $M_f(t)$ must be regular functions on all of $Y\times P$. Therefore, we obtain that $M_f(t)\in R[t]$ for any $f\in H^0(S_K,\Vcal_I(\lambda))$. Since $R$ is generated by homogeneous elements, we have shown the following theorem.

\begin{theorem}\label{thm-Rtriv-int}
The extension $R_{\triv}\subset R$ is integral.
\end{theorem}

\subsection{Existence of automorphic forms of trivial-type}

We finish this section by proving Theorem \ref{thm-triv-lambda}. Assume that $f\in H^0(S_K,\Vcal_I(\lambda))$ is a nonzero automorphic form of weight $\lambda$ and level $K$. We can find a connected component $S\subset S_K$ such that $f|_S$ is nonzero. Write $R$, $R_{\triv}$ for the rings attached to $S$ as defined in section \ref{sec-rings-R}. Since the extension $R_{\triv}\subset R$ is integral, there is a monic polynomial $P(t)\in R[t]$ such that $P(f|_S)=0$. Furthermore, by the proof of Theorem \ref{thm-triv-lambda}, we can assume that the degree of $P$ is less than $|L(\FF_p)|$. Write
\begin{equation}
    P(t)=t^n + a_1 t^{n-1}+ a_2 t^{n-2}+ \dots + a_0
\end{equation}
with $a_0,\dots,a_{n-1}\in R_{\triv}$. Since $f^n$ is a homogeneous element of weight $n\lambda$, some of the coefficients $a_i$ must have a nonzero homogeneous component of weight $d\lambda$ for some $1\leq d\leq n$. In particular, this shows that $H^0(S,\Vcal_I(d\lambda))_{\triv}$ is nonzero. This terminates the proof of Theorem \ref{thm-triv-lambda}.

\subsection{Trivial-type stable base locus}
We will need a version of Corollary \ref{cor-pb-nonzero} for the trivial-part of automorphic vector bundles. Define the trivial-part base locus $B_I(\lambda)_{\triv}$ of the vector bundle $\Vcal_I(\lambda)$ as the set of points $x\in S_K$ such that for any $f\in H^0(S_K,\Vcal_I(\lambda))_{\triv}$, the section $f$ vanishes at $x$. The set $B_I(\lambda)_{\triv}$ is a closed subscheme of $S_K$ which contains the usual base locus $B_I(\lambda)$. Similarly, define the trivial-part stable base locus $\BB_I(\lambda)_{\triv}$ of $\Vcal_I(\lambda)$ as the intersection
\begin{equation}
    \BB_I(\lambda)_{\triv}\colonequals \bigcap_{n\geq 1} B_I(n\lambda)_{\triv}. 
\end{equation}
We show the following:
\begin{proposition}
For all $\lambda\in X^*(T)$, we have $\BB_I(\lambda)_{\triv}=\BB_I(\lambda)$.
\end{proposition}

\begin{proof}
It suffices to shows that $\BB_I(\lambda)_{\triv}\subset \BB_I(\lambda)$. Let $x\in \BB_I(\lambda)_{\triv}$ and let $f\in H^0(S_K,\Vcal_I(d\lambda))$ be a section. We need to show that $f$ vanishes at $x$. We can find $a_0,\dots,a_{n-1}\in R_{\triv}$ such that 
\begin{equation}
    f^n + a_1 f^{n-1}+ a_2 f^{n-2}+ \dots + a_0 = 0.
\end{equation}
Projecting this relation onto the homogeneous part of weights $\NN \lambda$, we find a similar relation with the additional condition that $a_i\in H^0(S_K,\Vcal_I(d_i\lambda))_{\triv}$ for some $d_i\geq 1$. By assumption, all $a_i$ vanish at $x$. Hence the same is true for $f$. The result follows. 
\end{proof}

Finally, we are able to prove the technical lemma below, which will be necessary in the proof of the main theorem. We retain all notations and assumptions of Corollary \ref{cor-pb-nonzero}.

\begin{corollary}\label{cor-pb-nonzero-triv}
Assume that $H^0(X_G,\Vcal_I(\lambda))\neq 0$. Then there exists $m\geq 1$ and a nonzero trivial-type Siegel automorphic form $f\in H^0(X_G,\Vcal_I(m\lambda))_{\triv}$ such that the pullback
\[\widetilde{u}^*(f) \in H^0(X_H, \widetilde{u}^*(\Vcal_I(m\lambda)))\]
is nonzero.
\end{corollary}

\begin{proof}
The proof is exactly the same as Corollary \ref{cor-pb-nonzero}. Since $H^0(X_G,\Vcal_I(\lambda))\neq 0$ and $\BB_I(\lambda)_{\triv}=\BB_I(\lambda)$, we deduce that the trivial-part stable base locus $\BB_I(\lambda)_{\triv}$ does not intersect $S_K^{\ord}$. Since the ordinary locus of Hilbert--Blumenthal Shimura varieties is nonempty, we can choose an ordinary point $x\in \widetilde{u}(X_H)$. As $x\notin \BB_I(\lambda)_{\triv}$, we can find $m\geq 1$ and a section $f\in H^0(X_G,\Vcal_I(m\lambda))_{\triv}$ which does not vanish at $x$. The result follows.
\end{proof}

\section{Main result}

\subsection{The Cone Conjecture}

Let $X$ be a $k$-scheme endowed with a morphism $\zeta\colon X\to \GZip^\mu$ (such a map is sometimes referred to as a period map). We make the following assumptions on $X$:
\begin{assumption} \ \label{assumeX}
\begin{assertionlist}
    \item $X$ is a proper $k$-scheme,
    \item $\zeta$ is smooth and surjective on each connected component.
\end{assertionlist}
\end{assumption}
For example, let $S_K^{\Sigma}$ be a smooth toroidal compactification of a Hodge-type Shimura variety at a place of good reduction and $\zeta^{\Sigma}_K\colon S_K^{\Sigma}\to \GZip^\mu$ be the extension of Zhang's morphism (see section \ref{sec-tor}). Then, the pair $(S_K^{\Sigma},\zeta^\Sigma)$ satisfies the above assumption, as explained in \cite[\S 2.2]{Goldring-Koskivirta-divisibility}. One can also weaken the assumption that $X$ is proper and ask only that the strata of length $1$ in $\Flag(X)$ are proper (or even pseudo-complete), see \cite[\S 2.2]{Goldring-Koskivirta-divisibility}.

For each $\lambda\in X^*(T)$, write again $\Vcal_I(\lambda)$ for the vector bundle $\zeta^*(\Vcal_I(\lambda))$ on $X$. We define the cone of $X$ (denoted by $C_X$) as the set of $\lambda\in X^*(T)$ such that $\Vcal_I(\lambda)$ admits a nonzero global section on $X$, in other words:
\begin{equation}
    C_X\colonequals \left\{ \lambda\in X^*(T) \ \relmiddle| \  H^0(X,\Vcal_I(\lambda))\neq 0 \right\}.
\end{equation}
Write $\Ccal_X$ for the saturation of $C_X$.
\begin{conjecture}[Cone Conjecture]\label{conjX}
Under Assumption \ref{assumeX}, one has $\Ccal_X = \Ccal_{\zip}$.
\end{conjecture}

In particular, this conjecture applies to the pair $(S_K^{\Sigma},\zeta^\Sigma)$ as explained above. Furthermore, Koecher's principle (Theorem \ref{thm-koecher}) implies that $H^0(S_K,\Vcal_I(\lambda))=H^0(S^{\Sigma}_K,\Vcal^{\Sigma}_I(\lambda))$, so the above conjecture also applies to the Shimura variety $S_K$ itself, even though $S_K$ may not be proper over $k$. When $X=S_K$, the cone $\Ccal_X$ is simply the cone $\Ccal(\overline{\FF}_p)$ defined in section \ref{sec-autom-vb}.

\subsection{Siegel-type Shimura varieties}

We now specialize to the case of Siegel-type Shimura varieties, and set $G=\GSp_{2n,\FF_p}$. We retain all notations of section \ref{Siegel-type-sec}. The main result of this paper is the following:

\begin{theorem}\label{main-thm}
Conjecture \ref{conjX} holds for Siegel-type Shimura varieties $\overline{\Acal}_{n,K}$. In other words, one has 
\begin{equation}
\Ccal(\overline{\FF}_p)=\Ccal_{\zip}.    
\end{equation}
\end{theorem}

\begin{rmk} \ 
\begin{assertionlist}
    \item This theorem can be viewed as a cohomology vanishing result in degree $0$, in the sense that it shows that $H^0(\overline{\Acal}_{n,K},\Vcal_I(\lambda))=0$ whenever $\lambda$ lies outside the cone $\Ccal_{\zip}$. 
\item Note that it is possible that $H^0(\overline{\Acal}_{n,K},\Vcal_I(\lambda))=0$ even when $\lambda\in \Ccal_{\zip}$, but for any $\lambda\in \Ccal_{\zip}$, there exists $m\geq 1$ such that $H^0(\GZip^\mu,\Vcal_I(m\lambda))\neq 0$ (and in particular $H^0(\overline{\Acal}_{n,K},\Vcal_I(m\lambda))\neq 0$). This shows that Theorem \ref{main-thm} is the best possible cohomology vanishing result in degree zero that holds for all multiples of characters.
\item Theorem \ref{main-thm} reduces the vanishing of degree zero cohomology of $\Vcal_I(\lambda)$ to the case of $\GZip^\mu$, and this question can be reformulated in terms of representation theory using Theorem \ref{thm-H0-zip}. It is however difficult in general to determine exactly the form of the cone $\Ccal_{\zip}$.
\end{assertionlist}    
\end{rmk}

As preparation for the proof, we set some notation. Let $\lambda=(a_1,\dots,a_n,b)\in X^*(T)$ and let $V_I(\lambda)$ be the representation of $L$ attached to $\lambda$. The non-positive part of $V_I(\lambda)$ (defined in \eqref{nonposV}) can be written as 
\begin{equation}
    V_I(\lambda)_{\leq 0} = \bigoplus_{\substack{\chi=(a'_1,\dots,a'_n,b') \\ a'_n\leq 0}} V_I(\lambda)_{\chi}
\end{equation}
where we denoted by $\chi=(a'_1,\dots,a'_n,b')\in X^*(T)$ the coordinates of $\chi$. By Theorem \ref{thm-H0-zip}, the space $H^0(\GZip^\mu,\Vcal_I(\lambda))$ can be written as the intersection $V_I(\lambda)^{L(\FF_p)} \cap V_I(\lambda)_{\leq 0}$. Note that in this case $W_L\simeq \Sfr_n$ is naturally a subgroup of $L(\FF_p)$, so any element of $H^0(\GZip^\mu,\Vcal_I(\lambda))$ is invariant under the action of $W_L$ on $V_I(\lambda)$. Since $W_L\simeq \Sfr_n$ permutes the first $n$ coordinates of $X^*(T)\subset \ZZ^{n+1}$, it follows that any element of $H^0(\GZip^\mu,\Vcal_I(\lambda))$ lies in the following smaller subspace
\begin{equation}
    V_I(\lambda)_{\aut} \colonequals  \bigoplus_{\substack{\chi=(a'_1,\dots,a'_n,b') \\ a'_1,\dots,a'_n\leq 0}} V_I(\lambda)_{\chi}
\end{equation}
We call $V_I(\lambda)_{\aut}$ the "automorphic part" of the representation $V_I(\lambda)$. By the above, any automorphic form $f$ on $\GZip^\mu$ must satisfy $f\in V_I(\lambda)_{\aut}$. We will prove in the next section that automorphic forms on Siegel modular varieties must satisfy a similar condition, by using a Hilbert--Blumenthal Shimura variety embedded into $\overline{\Acal}_{n,K}$. As preparation for this, we make the following remarks. Define the "vanishing part" of $V_I(\lambda)$ as the direct sum
\begin{equation}
     V_I(\lambda)_{\van} \colonequals   \bigoplus_{\substack{\chi=(a'_1,\dots,a'_n,b') \\ \textrm{some } a'_i>0}} V_I(\lambda)_{\chi}.
\end{equation}
Hence we have a decomposition $V_I(\lambda)=V_I(\lambda)_{\aut}\oplus V_I(\lambda)_{\van}$. Write $\pr_{\van}\colon V_I(\lambda)\to V_I(\lambda)_{\van}$ for the projection map onto $V_I(\lambda)_{\van}$ afforded by this decomposition. By the above discussion, the space $H^0(\GZip^\mu,\Vcal_I(\lambda))$ sits in an exact sequence
\begin{equation}
\xymatrix@M=7pt{
  0\ar[r] & H^0(\GZip^\mu,\Vcal_I(\lambda)) \ar[r] &V_I(\lambda)^{L(\FF_p)} \ar[r]^{\pr_{\van}}  & V_I(\lambda)_{\van}.
}
\end{equation}
Here, we denoted again by $\pr_{\van}$ the composition of $\pr_{\van}$ and the inclusion of $V_I(\lambda)^{L(\FF_p)}$ into $V_I(\lambda)$. Therefore, we have the following obvious reformulation of the cone $C_{\zip}$:

\begin{proposition}\label{Czip-reformul}
A character $\lambda\in X^*(T)$ lies in $C_{\zip}$ if and only if the map
\begin{equation}
    \pr_{\van}\colon V_I(\lambda)^{L(\FF_p)} \to V_I(\lambda)_{\van}
\end{equation}
is not injective.
\end{proposition}

\subsection{Approximations}

In the paper \cite{Koskivirta-vanishing-Siegel}, we gave a partial result regarding cohomology vanishing of automorphic bundles. 

\begin{theorem}\label{thm-approx}
If $f\in H^0(\overline{\Acal}_{n,K}, \Vcal_I(\lambda))$ is a nonzero automorphic form of weight $\lambda=(a_1,\dots,a_n,b)$, we have: 
\begin{equation}\label{ineq-thm} 
\sum_{i=1}^j a_i + \frac{1}{p} \sum_{i=j+1}^{n} a_i \leq 0 \quad \textrm{ for all }j=1,\dots, n.
\end{equation}
\end{theorem}

In other words, if we define by $\Ccal_{\appro}$ the cone of characters $\lambda\in X^*(T)$ satisfying the above $n$ inequalities \eqref{ineq-thm}, then one has the containment $\Ccal(\overline{\FF}_p)\subset \Ccal_{\appro}$. This result was predicted by the Cone Conjecture. Indeed, we showed in \cite{Goldring-Koskivirta-GS-cone} that the cone $\Ccal_{\zip}$ satisfies such inequalities for very general pairs $(G,\mu)$ (see \loccit \S 4.3 for details). Hence, Theorem \ref{thm-approx} was already good evidence for the veracity of Conjecture \ref{conjX} (at least in the case of $\overline{\Acal}_{n,K}$).

\subsection{\texorpdfstring{The case $n=3$}{}}

In the case of $\overline{\Acal}_{n,K}$ for $n=3$, we gave a description of the cone $\Ccal_{\zip}$ in \cite{Koskivirta-automforms-GZip}. Specifically, we showed that it is defined inside $X^*_{+,I}(T)$ by two inequalities:
\begin{equation}
  \Ccal_{\zip} = \{ (a_1,a_2,a_3,b)\in X^*_{+,I}(T) \mid \ p^2 a_1+a_2+pa_3\leq 0 \ \textnormal{ and } \  pa_1+p^2 a_2+a_3\leq 0 \}.   
\end{equation}
Furthermore, Goldring and the author showed in \cite{Goldring-Koskivirta-divisibility} that Conjecture \ref{conjX} holds for $n=3$ and $p\geq 5$ for any pair $(X,\zeta_X)$ satisfying Assumption \ref{assumeX}. The proof of \loccit is highly computational and uses the notion of intersection-sum cones to produce inductively vanishing results for certain strata in the flag space $\Flag(X)$. The assumption $p\geq 5$ is essential in this method. In the case $X=\overline{\Acal}_{3,K}$, we deduce from Theorem \ref{main-thm} that the equality $\Ccal(\overline{\FF}_p)=\Ccal_{\zip}$ holds without any assumption on $p$. Hence, we obtain the following vanishing result for coherent cohomology of automorphic vector bundles (for all $p$):

\begin{theorem}
Let $\lambda=(a_1,a_2,a_3,b)\in X^*(T)$ and assume that $p^2 a_1+a_2+pa_3 > 0$ or that $pa_1+p^2 a_2+a_3 > 0$, Then one has
    \begin{equation}
H^0(\overline{\Acal}_{3,K},\Vcal_I(\lambda)) = 0.
    \end{equation}
\end{theorem}

Although highly tedious and computational, one advantage of the method of intersection-sum cones used in \cite{Goldring-Koskivirta-divisibility} to prove the above result (when $p\geq 5$) is that it also provides divisibility results for automorphic forms. Indeed, we showed in \loccit the following result. Let $\Ha_{\alpha_1}$ be the partial Hasse invariant explained in \loccit \S 3.3 (the vanishing locus of $\Ha_{\alpha_1}$ is one of the Zariski closure of one of the three codimension one strata in $\Flag(S_K)$). We showed:

\begin{theorem}\label{thm-divSp6}
Assume $p\geq 5$. Let $f\in H^0(\overline{\Acal}_{3,K},\Vcal_{I}(\lambda))$ be an automorphic form, viewed as a section of $\Vcal_{\flag}(\lambda)$ on $\Flag(S_K)$, and assume that $\lambda=(a_1,a_2,a_3,b)\in X^*(T)$ satisfies $p^2a_1+pa_2+a_3>0$. Then $f$ is divisible by the partial Hasse invariant $\Ha_{\alpha_1}$.
\end{theorem}

In contrast, the method used in the present paper does not give any information about the vanishing of cohomology on smaller Ekedahl--Oort strata or smaller strata in the flag space $\Flag(S_K)$. Therefore we cannot prove such a divisibility result.

\subsection{Proof of the Cone Conjecture}

In this section, we prove Theorem \ref{main-thm}. Since $\Ccal_{\zip}\subset \Ccal(\overline{\FF}_p)$, it suffices to show that if $H^0(\overline{\Acal}_{n,K},\Vcal_I(\lambda))\neq 0$, then there exists a nonzero element in $H^0(\GZip^\mu,\Vcal_I(d\lambda))$ for a certain $d\geq 1$. First, by Theorem \ref{thm-triv-lambda}, there exists $1\leq d \leq N$ such that 
\begin{equation}\label{triv-nonzero}
    H^0(\overline{\Acal}_{n,K},\Vcal_I(d\lambda)_{\triv}) \neq 0.
\end{equation}
We choose an embedding of a Hilbert--Blumenthal Shimura variety $\widetilde{u}\colon X_{\mathbf{F},K'}\to \overline{\Acal}_{n,K}$, as explained in section \ref{hb-var-sec}. We impose that $p$ is split in $\mathbf{F}$. Write $u\colon H\to G$ the morphism of $\FF_p$-groups given in \ref{u-embedQ}. For any character $\chi\in X^*(T_H)$, recall that we denote by $\Lcal(\chi)$ the attached line bundle on $X_{\mathbf{F},K'}$ (see section \ref{hb-var-sec}). We show the following result:
\begin{proposition}\label{prop-triv-zip}
Let $\lambda\in X^*(T)$ and $f\in H^0(\overline{\Acal}_{n,K},\Vcal_I(\lambda))_{\triv}$ a nonzero automorphic form of trivial-type. Assume that $\widetilde{u}^*(f) \in H^0(X_{\mathbf{F},K'}, \widetilde{u}^*(\Vcal_I(\lambda)))$ is nonzero. Then we have
\begin{equation}
    H^0(\GZip^\mu,\Vcal_I(\lambda))\neq 0.
\end{equation}
\end{proposition}

\begin{proof}
By Proposition \ref{Czip-reformul}, it suffices to show that the map
\begin{equation}\label{prvan-inj}
    \pr_{\van}\colon V_I(\lambda)^{L(\FF_p)} \to V_I(\lambda)_{\van}
\end{equation}
is not injective. For a contradiction, assume that \eqref{prvan-inj} is injective. By the explicit form of $u$ (see end of section \ref{hb-var-sec}), it induces an isomorphism $u\colon T_H\to T$ and a homomorphism $u\colon B_H\to B$. By assumption, $\widetilde{u}^*(f)$ is a nonzero section of $\widetilde{u}^*(\Vcal_I(\lambda))$. By pullback via $u$, we van view $V_I(\lambda)$ as a $T_H$-representation. Since $T_H\simeq T$, the weight decomposition $V_I(\lambda)=\bigoplus_{\chi\in X^*(T)} V_I(\lambda)_\chi$ gives rise to a similar decomposition of $\widetilde{u}^*(\Vcal_I(\lambda))$ into vector bundles as below:
\begin{equation}
    \widetilde{u}^*(\Vcal_I(\lambda))=\bigoplus_{\chi\in X^*(T)} \Vcal(\chi)
\end{equation}
where $\Vcal(\chi)$ is the vector bundle on $X_{\mathbf{F}, K'}$ corresponding to the $T_H$-representation $V_I(\lambda)_\chi$. In particular, $\Vcal(\chi)$ is a direct sum of a finite number of copies of the line bundle $\Lcal(\chi)$. 

Similarly, view the $T$-representation $V_I(\lambda)_{\van}$ as a $T_H$-representation via $u\colon T_H\to T$. It induces by functoriality a vector bundle $\Vcal(\lambda)_{\van}$ on $X_{\mathbf{F}, K'}$ (which does not exist on $\overline{\Acal}_{n,K}$, since $V_I(\lambda)_{\van}$ is not an $L$-representation). Since the projection $\pr_{\van}\colon V_I(\lambda)\to V_I(\lambda)_{\van}$ is a morphism of $T$-representations, it induces by functoriality a morphism of vector bundles on $X_{\mathbf{F}, K'}$:
\begin{equation}
    \pr_{\van}\colon \widetilde{u}^*(\Vcal_I(\lambda))\to \Vcal_I(\lambda)_{\van}.
\end{equation}
Similarly, $\widetilde{u}$ induces an embedding $\widetilde{u} \colon X^{\ord}_{\mathbf{F},K'}\to S^{\ord}_{K}$. Any $T_H(\FF_p)$-representation gives rise to a vector bundle on the scheme $X^{\ord}_{\mathbf{F},K'}$. The vector bundle on $X^{\ord}_{\mathbf{F},K'}$ attached to $V_I(\lambda)^{L(\FF_p)}$ coincides with $\widetilde{u}^*(\Vcal_I(\lambda)_{\triv})$. Moreover, the map 
\begin{equation}
\pr_{\van}\colon V_I(\lambda)^{L(\FF_p)}\to V_I(\lambda)_{\van}    
\end{equation}
is an injective morphism of $T(\FF_p)$-representations. Composing with the isomorphism $T_H(\FF_p)\to T(\FF_p)$, it gives rise to an injective morphism of vector bundles over $X^{\ord}_{\mathbf{F},K'}$:
\begin{equation}
    \pr_{\van}\colon \widetilde{u}^*(\Vcal_I(\lambda)_{\triv}) \to \Vcal_I(\lambda)_{\van}.
\end{equation}
In particular, the induced morphism on the spaces of sections over $X^{\ord}_{\mathbf{F},K'}$ is also injective:
\begin{equation}\label{inj-morph-Xord}
        \pr_{\van}\colon H^0(X^{\ord}_{\mathbf{F},K'}, \widetilde{u}^*(\Vcal_I(\lambda)_{\triv})) \to H^0(X^{\ord}_{\mathbf{F},K'}, \Vcal_I(\lambda)_{\van}).
\end{equation}
By assumption, $\widetilde{u}^*(f)$ is a nonzero section of $\widetilde{u}^*(\Vcal_I(\lambda)_{\triv})$ over $X^{\ord}_{\mathbf{F},K'}$. However, since $f\in H^0(\overline{\Acal}_{n,K},\Vcal_I(\lambda))$ is a global section, $\widetilde{u}^*(f)$ lies in the space
\begin{equation}
    H^0(X_{\mathbf{F},K'}, \widetilde{u}^*(\Vcal_I(\lambda))) = \bigoplus_{\chi\in X^*(T)} H^0(X_{\mathbf{F},K'},\Vcal(\chi)).
\end{equation}
Since the cone conjecture holds for Hilbert--Blumenthal Shimura varieties, the space of global sections $H^0(X_{\mathbf{F},K'},\Vcal(\chi))$ is zero whenever $\chi=(a_1,\dots,a_n,b)$ and $a_i>0$ for some $1\leq i \leq n$. Hence, $\widetilde{u}^*(f)$ maps to zero in the space $H^0(X^{\ord}_{\mathbf{F},K'},\Vcal_I(\lambda)_{\van})$. This contradicts the injectivity of the map \eqref{inj-morph-Xord} and terminates the proof of Proposition \ref{prop-triv-zip}.
\end{proof}

Finally, we complete the proof of our main result Theorem \ref{main-thm}. Assume that the space $H^0(\overline{\Acal}_{n,K},\Vcal_I(\lambda))$ is nonzero. Then by \ref{cor-pb-nonzero-triv}, there exists $d\geq 1$ and a trivial-type automorphic form $f\in H^0(\overline{\Acal}_{n,K},\Vcal_I(d\lambda))_{\triv}\neq 0$ such that the pullback $\widetilde{u}^*(f) \in H^0(X_{\mathbf{F},K'}, \widetilde{u}^*(\Vcal_I(\lambda)))$ is nonzero. By Proposition \ref{prop-triv-zip} above, we deduce $H^0(\GZip^\mu,\Vcal_I(d\lambda))\neq 0$. This terminates the proof of Theorem \ref{main-thm}.

\begin{rmk}\label{rmk-precise-weight}
Assume that Question \ref{quest-Hilbert} has a positive answer. Then we can prove a slightly more precise version of the Cone Conjecture. Namely, if $H^0(\overline{\Acal}_{n,K},\Vcal_I(\lambda))$ is nonzero, then we would be able to deduce that there exists $1\leq d \leq |L(\FF_p)|$ such that  $H^0(\GZip^\mu,\Vcal_I(d\lambda))$ is nonzero.
\end{rmk}

\bibliographystyle{test}
\bibliography{biblio_overleaf}

\noindent
Jean-Stefan Koskivirta\\
Department of Mathematics, Faculty of Science, Saitama University, 
255 Shimo-Okubo, Sakura-ku, Saitama City, Saitama 338-8570, Japan \\
jeanstefan.koskivirta@gmail.com

\end{document}